\documentclass[a4paper, 11pt]{article}

\usepackage[utf8]{inputenc}
\usepackage[T1]{fontenc}
\usepackage[english]{babel}
\usepackage{authblk}
\usepackage{hyperref}  
\usepackage{amsfonts, amsthm, amsmath, amssymb} 
\usepackage{lineno}
\usepackage{old-arrows}  
\usepackage{array}  
\usepackage{subcaption}  
\usepackage{color}
\usepackage{pgfplots}
\pgfplotsset{compat=1.14}
\usepackage{stmaryrd}
\usepackage[margin=1.1in, showframe=false]{geometry}  
\usepackage{float}
\usepackage{enumitem}  
\usepackage{pdfpages}
\usepackage{tcolorbox}  
\usepackage{comment}
\usepackage{graphicx}
\usepackage{graphics}
\usepackage{xspace}
\usepackage[vlined, ruled, dotocloa, english, onelanguage, linesnumbered]{algorithm2e}
\usepackage[export]{adjustbox}  

\definecolor{midnight}{RGB}{44, 62, 80}
\definecolor{asbestos}{RGB}{127, 140, 141}
\definecolor{clouds}{RGB}{236, 240, 241}
\definecolor{auroragreen}{RGB}{80, 184, 103}
\definecolor{electron}{RGB}{9, 132, 227}
\definecolor{cornflower}{RGB}{84, 109, 229}  

\hypersetup{
	colorlinks=true,
	citecolor=cornflower,
	linkcolor=asbestos
}

\newcommand{\cc}[1]{\mathcal{#1}}  
\newcommand{\cb}[1]{\mathbb{#1}}  
\newcommand{\csf}[1]{\textsf{#1}}  
\newcommand{\csmc}[1]{\textsc{#1}}  
\newcommand{\ctt}[1]{\texttt{#1}}  
\newcommand{\nf}[1]{{\normalfont #1}}  

\newcommand{\eg}{\emph{e.g.}\xspace}  
\newcommand{\ie}{\emph{i.e.}\xspace}  
\newcommand{\wrt}{{\emph{w.r.t.}}\xspace}  

\newcommand{\card}[1]{\vert #1 \vert}  
\newcommand{\U}{V}  
\newcommand{\pow}[1]{\mathbf{2}^{#1}}  
\renewcommand{\setminus}{\smallsetminus}

\renewcommand{\max}{{\normalfont \csf{max}_{\subseteq}}}  
\renewcommand{\min}{{\normalfont \csf{min}_{\subseteq}}}  

\renewcommand{\P}{\textbf{\csf{P}}}  
\newcommand{\NP}{\textbf{\csf{NP}}}  
\newcommand{\poly}{\csf{poly}}  

\newcommand{\A}{\cc{A}}  
\newcommand{\B}{\cc{B}}  
\newcommand{\Bp}{\B^+}  
\newcommand{\Bm}{\B^-}  

\newcommand{\cl}{\phi}  
\newcommand{\cs}{\cc{C}} 
\newcommand{\is}{\Sigma}  
\newcommand{\Ext}{{\normalfont \csf{Ext}}}  

\newcommand{\M}{\cc{M}}  

\DeclareMathOperator{\ftr}{\uparrow}  
\DeclareMathOperator{\idl}{\downarrow}  
\DeclareMathOperator{\imp}{\rightarrow}  

\theoremstyle{plain}
\newtheorem{theorem}{Theorem}  
\newtheorem{proposition}{Proposition}  
\newtheorem{corollary}{Corollary}  
\newtheorem{lemma}{Lemma}  
\newtheorem{definition}{Definition}  

\theoremstyle{definition}

\theoremstyle{remark}
\newtheorem{example}{Example}  
\newtheorem{remark}{Remark}  

\tcbuselibrary{many}

\newtcolorbox{problem}[1]{
	colframe=clouds,
	titlerule style=\color{black},
	colback=white,
	fonttitle=\color{black},
	arc=1mm,
	enhanced,
	attach boxed title to top left={xshift=0.5cm, yshift=-3.7mm},
	boxed title style={colframe=white},
	title=\csmc{#1},
	colbacktitle=white,
}

\newtcolorbox{emphasize}{
	colback=clouds,
	colframe=clouds,
	arc=0mm,
	left=0mm,
	right=0mm,
	top=0mm,
	bottom=0mm,
	before skip=5mm,
	after skip=5mm	
}

\newcounter{question}
\newenvironment{question}
{\refstepcounter{question}\par\medskip\noindent	\textbf{Question~\thequestion.} \it}
{\medskip}

\newcommand{\Problem}[3]{
	\begin{problem}{#1}
		\begin{tabular}{l p{0.7\textwidth}}
			\textit{Input: \hspace{0.7em}} & #2 \\ 
			\textit{Output: \hspace{0.7em}} & #3 \\	
		\end{tabular}
	\end{problem}
}

\title{Hierarchical decompositions of implicational bases for the enumeration of 
meet-irreducible elements}
\author[1]{Lhouari Nourine}
\author[1, 2]{Simon Vilmin}

\affil[1]{Universit\'e Clermont-Auvergne, CNRS, Mines de Saint-\'Etienne, 
	Clermont-Auvergne-INP, LIMOS, 63000 Clermont-Ferrand, France.}
\affil[2]{Universit\'e de Lorraine, CNRS, LORIA, F-54000, France.}

\begin{document}
	
\maketitle

\begin{abstract}
	We are interested in the problem of translating between two representations 
	of closure systems, namely implicational bases and meet-irreducible elements.
	Albeit its importance, the problem is open.
	Motivated by this problem, we introduce \emph{splits} of an implicational 
	base.
	It is a partitioning operation of the implications which we apply recursively 
	to obtain a binary tree representing a decomposition of the implicational base.
	We show that this decomposition can be conducted in polynomial time and space 
	in the size of the input implicational base.
	In order to use our decomposition for the translation task, we focus on the 
	case of \emph{acyclic} splits.
 	In this case, we obtain a recursive characterization of the meet-irreducible elements 
 	of the associated closure system.
	We use this characterization and hypergraph dualization to derive new results 
	for the translation problem in acyclic convex geometries.
	
\end{abstract}

\section{Introduction}
\label{sec:trad:intro}

Finite closure systems over a (finite) ground set are set systems containing the 
ground set and closed under intersection.
When ordered by inclusion, they are also known as (closure) lattices 
\cite{davey2002introduction, gratzer2011lattice}.
These structures are well-known in mathematics and computer science.
They show up in Knowledge Space Theory (KST) \cite{doignon2012knowledge}, database 
theory \cite{demetrovics1992functional, mannila1992design}, 
propositional logic \cite{kautz1993reasoning, khardon1995translating}, Formal 
Concept Analysis (FCA) \cite{ganter2012formal}, or
argumentation frameworks \cite{dung1995acceptability, elaroussi2021lattice} for 
example.

Albeit ubiquitous, closure systems suffer from their size, which can be 
exponential in the size of their ground set.
For this reason, numerous research works have been conducted over the last decades 
to 
construct space efficient representations of lattices 
\cite{ausiello1986minimal, bertet2010multiple, ganter2012formal, 
guigues1986familles, 	
	habib2018representation, khardon1995translating, mannila1992design, 
	markowsky1975factorization, wild1994theory}.
The surveys \cite{bertet2018lattices, wild2017joy} are also recent witnesses of 
the 
importance and the relevance 
of compactly representing closure systems.

Among all possible representations, there are two prominent candidates:  
\emph{implications} and \emph{meet-irreducible elements}.
An implication is a mathematical expression $A \imp B$, where $A$ and $B$ are 
subsets of 
the ground set, modeling a causality relation between $A$ and $B$ in the closure 
system: \textit{``if a set includes $A$, it must also include $B$''}.
Every closure system $\cs$ over some ground set $\U$ can be represented by a set 
$\is$ of 
implications called an \emph{implicational base}.
Dually, every set of implications gives birth to a closure system 
\cite{wild1994theory}.
As several implicational bases can represent the same closure system, numerous 
bases 
with \textit{``good''} properties have been studied.
Among them, the Duquenne-Guigues base \cite{guigues1986familles} being minimum 
(the least 
number of implications) or the canonical direct base \cite{bertet2010multiple} 
have 
attracted much attention.
More recently, Adaricheva \textit{et al.} \cite{adaricheva2014implicational, 
	adaricheva2017discovery, adaricheva2013ordered} have proposed refinements of 
the canonical direct base such as the $D$-base and the $E$-base.
Because of their simple nature, implications have been used under different shapes 
and 
names such as functional dependencies in databases 
\cite{demetrovics1992functional, 
	mannila1992design}, Horn functions in propositional logic 
	\cite{kautz1993reasoning, 
	khardon1995translating}, queries in KST \cite{doignon2012knowledge} or 
	attribute 
implications in FCA \cite{ganter2012formal, guigues1986familles} for 
instance.
A second way to compactly represent a closure system $\cs$ is its family of 
meet-irreducible elements $\M$. 
It is the unique minimal collection of sets from which the whole closure system 
can be recovered by taking intersections.
In Horn logic, meet-irreducible elements are called \emph{characteristic 
	models} \cite{khardon1995translating, kautz1993reasoning} for they 
completely identify a given Horn function. 
Moreover, they appear in the poset of irreducibles 
\cite{habib2018representation, markowsky1975factorization}, in the 
Armstrong relations in databases \cite{mannila1992design}, in the 
base of knowledge spaces \cite{doignon2012knowledge} or in 
the reduced context of FCA \cite{ganter2012formal}.

In this paper, we study the problem of translating between these two 
representations.
This problem is twofold. 
Either it asks to list the meet-irreducible elements of a closure system given by 
an 
implicational base, or vice-versa, to construct an implicational base from a set 
of 
meet-irreducible elements.
Due to the polyvalence of closure systems and their representations, the translation task appears in disguise in the areas mentioned above.
For example in databases, it connects with the question of finding functional dependencies in data, or to the problem of designing Armstrong relations for given dependencies \cite{mannila1992design}.
Another example stems from Knowledge Space Theory \cite{doignon2012knowledge}, where enumerating meet-irreducible relates to the enumeration of the base of a knowledge space from a family of queries.
Similar connections arise from propositional logic \cite{kautz1993reasoning,khardon1995translating} with the listing of characteristic models given a Horn CNF (and vice-versa), or from FCA \cite{ganter2012formal,guigues1986familles} when one seeks to describe a context by means of a small number of attribute implications.

In fact, the choice of the representation impacts the complexity of several problems, thus making the translation a crucial task.
For example, it is \NP-complete to decide whether an element belongs to a minimal 
generator of a 
closure system if the latter is given by an implicational base 
\cite{lucchesi1978candidate}.
When the closure system is represented by its meet-irreducible elements, we can 
answer the question in polynomial time \cite{bertet2018lattices}.
The complexity of recognizing a class of closure systems also depends on the 
representation.
For instance, it takes polynomial time to recognize convex geometries and 
join-semidistributivity from meet-irreducible elements \cite{edelman1985theory, 
nation2000unbounded}, but the complexity of the task given an implicational base depends 
on the implications at hand.
If the $D$-base is given, a recent paper of Adaricheva \emph{et al.}
\cite{adaricheva2022notes} shows that recognizing both classes can be done in polynomial 
time.
However, if one considers an arbitrary implicational base, these recognition problems 
become \csf{co}\NP-complete \cite{bichoupan2022complexity}.
Another example where the representation matters comes from propositional logic 
\cite{kautz1993reasoning}, where abductive reasoning can be conducted in 
polynomial time 
from meet-irreducible elements, while it is \NP-complete with implications.

Translating is also important to enjoy the most compact representation for a given 
closure system.
Indeed, implicational bases and meet-irreducible elements are generally much 
shorter than the closure systems they represent.
However, when comparing the two representations, there are cases where an 
implicational 
base has size exponential in the number of meet-irreducible elements, or dually, 
where 
the number of meet-irreducible elements can be exponential in the size of an 
implicational base \cite{kuznetsov2004intractability, mannila1994algorithms}.

\paragraph{Known results.}
We now review the principal results on the translation task.
It has attracted much attention during the last decades 
\cite{adaricheva2017discovery, babin2013computing, beaudou2017algorithms, 
	khardon1995translating, mannila1992design, wild1995computations}.
The surveys \cite{bertet2018lattices, wild2017joy} provide a detailed account of 
all the 
progresses made on this question.
Since the size of the output can be exponential in the size of the input, we 
express the
complexity results in terms of the combined size of the input and the output.
This is \emph{output-sensitive} complexity \cite{johnson1988generating}.

For completeness we discuss four representations for a closure system: 
implications, 
meet-irreducible elements, the closure system itself or the closure operator.
The closure operator is seen as a black-box oracle returning the smallest closed 
set including a given set.
We explain each direction of Figure \ref{fig:trad:hard-trad}, which summarizes 
hardness 
results about the translation task.
Numbers in the figure refer to the following explanations.

\begin{figure}[ht!]
	\centering 
	\includegraphics[scale=0.85, page=1]{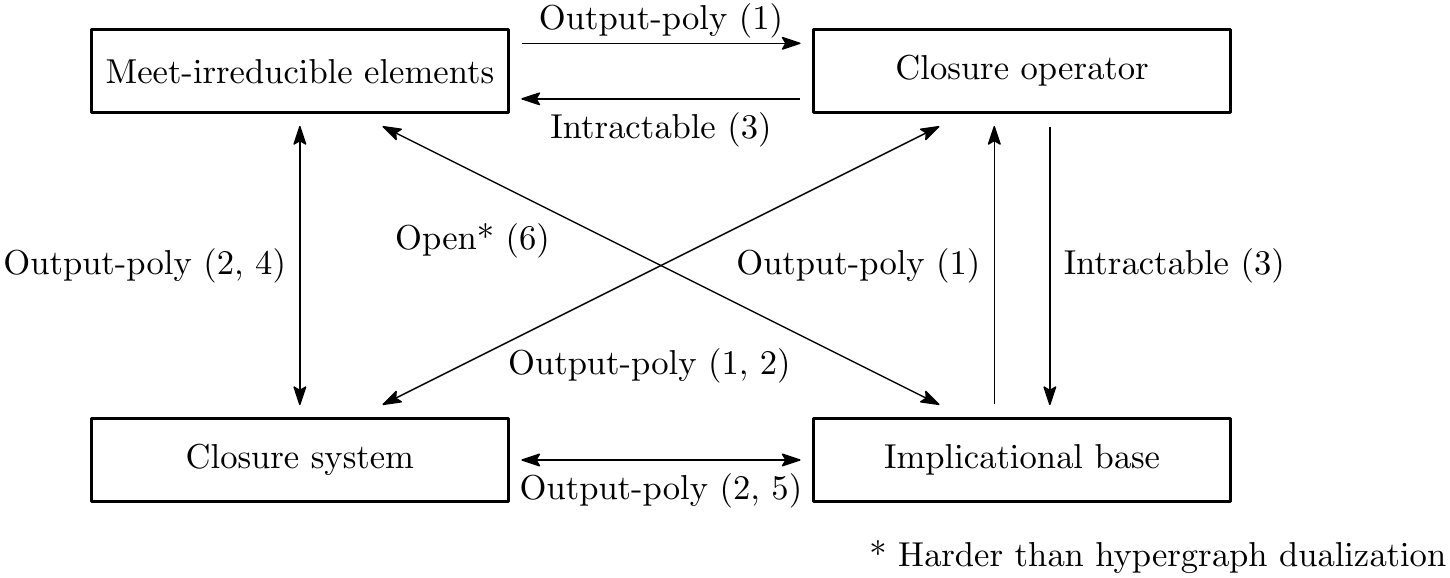}%
	\caption{The complexity of translating between the representations of a 
	closure 
		system.}
	\label{fig:trad:hard-trad}
\end{figure}

\emph{(1) From any representation to the closure operator}.
The closure operation can be simulated in polynomial-time from any other 
representation of the closure system, using intersections and the closure 
algorithm 
(or the forward chaining) \cite{ganter2012formal}.

\emph{(2) From any representation to the closure system}.
The whole closure system can be constructed in output-polynomial time from any 
other representation, with the help of well-known algorithms such as 
\ctt{NextClosure} \cite{ganter2012formal}.

\emph{(3) From the closure operator to meet-irreducible elements and 
implications}.
Lawler \textit{et al.} prove in \cite{lawler1980generating} that meet-irreducible 
elements or 
implications cannot be enumerated in output-polynomial time unless $\P = \NP$ from 
a 
closure oracle.

\emph{(4) From the closure system to its meet-irreducible elements}. 
It is sufficient to perform a traversal of the closed sets, and check for the 
meet-irreducible property.
This is done in (output)-polynomial time.

\emph{(5) From the closure system to an implicational base}.
To find a (minimum) implicational base, it is for instance possible to use the 
attribute-incremental approach of Duquenne and Obiedkov 
\cite{obiedkov2007attribute} 
in output-polynomial time.

\emph{(6) From an implicational base to meet-irreducible elements and vice-versa}.
Remark that undertaking the construction of the whole closure system as an 
intermediate 
will necessarily produce output-exponential time algorithms in the worst case.
In the landmark paper \cite{khardon1995translating}, written in the framework of 
Horn 
logic, these problems are called \csmc{CCM} for \emph{Computing Characteristic 
Models} 
and \csmc{SID} for \emph{Structure Identification}.
We keep these names for historicity.

\Problem{Meet-irreducible elements enumeration (CCM)}
{An implicational base $\is$ of a closure system $\cs$ over $\U$.}
{The meet-irreducible elements $\M$ of $\cs$.}

\Problem{Minimum implicational base identification (SID)}
{The family $\M$ of meet-irreducible elements of a closure system $\cs$ over $\U$.}
{A minimum implicational base $\is$ corresponding to $\cs$.}

In \cite{khardon1995translating}, the author considers right-optimum 
implicational bases 
(minimizing the right-hand sides of implications) 
and shows that both directions of the translation are equivalent.
Whether this equivalence also holds for minimum implicational bases is not clear 
as going 
from right-optimum to minimum is much easier than the other way around
\cite{ausiello1986minimal, shock1986computing}.
In any case, the task is already harder than enumerating the maximal independent 
sets of 
a hypergraph \cite{khardon1995translating}.
This latter problem, also known as \emph{hypergraph dualization}, is a famous open 
problem \cite{eiter1995identifying, fredman1996complexity}.
The best known algorithm for this task is the one of Fredman 
and Khachiyan \cite{fredman1996complexity}, running in output quasi-polynomial 
time.
Babin and Kuznetsov prove in \cite{babin2013computing} that it is 
\csf{co}\NP-complete to decide whether an implication belongs to a minimum 
implicational 
base from the meet-irreducible elements.
In \cite{kavvadias2000generating}, the authors state that co-atoms of a closure 
system 
cannot be enumerated in output-polynomial time unless $\P = \NP$.
In \cite{distel2011complexity}, it is shown that the minimal pseudo-closed sets of 
the 
Duquenne-Guigues basis cannot be enumerated in output-polynomial time unless $\P = 
\NP$ 
either.
More recently \cite{defrain2021translating}, it has been shown that \csmc{CCM} and 
\csmc{SID} are harder than hypergraph dualization, even in acyclic convex 
geometries.
In spite of these hardness results, the complexity of translating between 
meet-irreducible elements and implications remains unsettled.

On the positive side, finding the canonical direct base from the meet-irreducible 
elements (and vice-versa) is equivalent to hypergraph dualization
\cite{bertet2018lattices, bertet2010multiple, khardon1995translating}.
Adaricheva et al. \cite{adaricheva2017discovery} obtain similar results 
for the 
$D$-base.
More generally, exponential time algorithms have been designed, see \eg 
\cite{ganter2012formal, obiedkov2007attribute, mannila1992design, 
wild1995computations}.
In \cite{wild2000optimal}, Wild shows that \csmc{SID} can be solved in polynomial 
time in 
modular lattices.
The authors in \cite{beaudou2017algorithms} devise output-polynomial time 
algorithms for both \csmc{CCM} and \csmc{SID} in $k$-meet-semidistributive 
lattices.
Finally, it has been proved \cite{defrain2021translating} that \csmc{CCM} and 
\csmc{SID} 
are polynomially equivalent to hypergraph dualization in the class of \emph{ranked 
convex 
	geometries}.

\paragraph{Contributions and outline} 
We are mostly interested in the problem \csmc{CCM} in the class of acyclic convex 
geometries.
Convex geometries form an ubiquitous class of closure systems.
They arise from several mathematical objects such as graphs, hypergraphs, ordered sets, or points in the plane \cite{edelman1985theory, farber1986convexity, kashiwabara2010characterizations, korte2012greedoids}.
In particular, acyclic convex geometries are well-studied closure systems
\cite{adaricheva2013ordered, hammer1995quasi, wild1994theory, zanuttini2015proprietes}, lying in the 
	intersection of convex geometries and lower-bounded closure systems 
	\cite{adaricheva2014implicational, freese1995free}.
In acyclic convex geometries, \csmc{CCM} and \csmc{SID} are harder than hypergraph dualization \cite{defrain2021translating}.
However, they also contain distributive closure systems, in which the translation can be solved efficiently.
As a consequence, acyclic convex geometries are an important class of systems to study in order to better understand the complexity of the translation task.
Our contribution is a step towards this direction.
By means of implicational bases, we seek to shed the light on the structure of closure systems, particularly acyclic ones, with respect to the problem of enumerating the meet-irreducible elements.

Let $\is$ be an implicational base for some closure system $\cs$ over $\U$.
We start with some preliminary definitions in Section \ref{sec:trad:prelim}.
Then, we give the following results:
\begin{enumerate}
	\item We introduce a partitioning operation of an implicational base called a 
	\emph{split}, inspired by \cite{dasgupta2016cost, libkin1993direct}.
	We use this operation to hierarchically decompose $\is$ and its associated 
	closure 
	system $\cs$.
	This part is detailed in Section \ref{sec:trad:split}.
	
	\item Section \ref{sec:trad:acyclic} is devoted to \emph{acyclic splits}:
	\begin{itemize}
		\item[(1)] We characterize $\cs$ with respect to this partitioning 
		operation, see 
		Subsection 	\ref{subsec:trad:acyc-split-cs}.
		
		\item[(2)] We derive a recursive characterization of the set of
		meet-irreducible elements $\M$ associated to $\cs$, see Subsection 
		\ref{subsec:trad:acyc-split-meet}.
		
		\item[(3)] We devise an algorithm solving \csmc{CCM} in the presence of 
		acyclic 
		splits.
		We highlight cases where this procedure performs in output-quasipolynomial 
		time 
		using the algorithm of Fredman and Khachiyan \cite{fredman1996complexity} 
		for 
		hypergraph dualization.
		This result includes ranked convex geometries as a particular case.
		This is Subsection \ref{subsec:trad:trad}.	
	\end{itemize}
\end{enumerate}

The paper gathers results communicated at the 21st conference ICTCS (for Section 
\ref{sec:trad:split}) and the 8th workshop FCA4AI (for Section 
\ref{sec:trad:acyclic}), 
without published proceedings.

\section{Preliminaries}
\label{sec:trad:prelim}

All the objects considered in this paper are \emph{finite}.
For more definitions about closure systems and implications, we refer the reader 
to 
\cite{bertet2018lattices}.
If $\U$ is a set, we refer to $\pow{\U}$ as the family of all subsets of $\U$.
Sometimes, and mostly in examples, we shall write the subset $\{u_1, \dots, u_k\}$ 
of 
$\U$ as the concatenation of its elements, that is $u_1 \dots u_k$.
The size of a set $\U' \subseteq \U$ is denoted by $\card{\U'}$.
Let $\cc{S}$ be a family of subsets of $\U$.
We say that $\cc{S}$ is \emph{simple} or an \emph{antichain} if for every $S_1, 
S_2 \in 
\cs$, $S_1 \nsubseteq S_2$.
Let $\U' \subseteq \U$.
The \emph{trace} of $\cc{S}$ on $\U'$, denoted by $\cc{S} \colon \U'$, is obtained 
by 
intersecting each element of $\cc{S}$ with $\U'$, that is $\cc{S} \colon \U' = \{S 
\cap 
\U' \mid S \in \cc{S} \}$.

\paragraph{Closure systems, closure operators}
Let $\U$ be a set.
A \emph{closure system} over $\U$ is a family $\cs$ of subsets of $\U$ such that 
$\U \in 
\cs$ and $C_1 \cap C_2 \in \cs$ for every $C_1, C_2 \in \cs$.
The sets in $\cs$ are called \emph{closed (sets)}.
When ordered by inclusion, the pair $(\cs, \subseteq)$ is a \emph{(closure) 
lattice}.
In this paper, we always assume that a closure system is equipped with this order.
Hence, we write $\cs$ to denote the lattice $(\cs, 
\subseteq)$.
Let $C_1, C_2 \in \cs$.
We say that $C_1$ and $C_2$ are \emph{comparable} if $C_1 \subseteq C_2$ or $C_2 
\subseteq C_1$.
We write $C_1 \subset C_2$ if $C_1 \subseteq C_2$ but $C_1 \neq C_2$.
We say that $C_2$ \emph{covers} $C_1$, denoted by $C_1 \prec C_2$, if $C_1 \subset 
C_2$ 
and there is no closed set $C \in \cs$ such that $C_1 \subset C \subset C_2$.
In this case, $C_2$ is a \emph{successor} of $C_1$ and $C_1$ a \emph{predecessor} 
of 
$C_2$.
Let $C \in \cs$.
The \emph{ideal} of $C$ in $\cs$, denoted $\idl C$ contains all the closed subsets 
of 
$C$, \ie $\idl C= \{C' \in \cs \mid C' \subseteq C\}$.
The \emph{filter} $\ftr C$ of $C$ in $\cs$ is defined dually with the closed 
supersets of 
$C$.
If $\cs'$ is a subset of $\cs$, the ideal of $\cs'$ is $\idl \cs' = \bigcup_{C' 
\in \cs'} 
\idl C'$ and its filter is $\ftr \cs' = \bigcup_{C' \in \cs'} \ftr C'$. 
A closed set $M$ is \emph{meet-irreducible} if $M = C_1 \cap C_2$ with $C_1, C_2 
\in \cs$ 
implies $M = C_1$ or $M = C_2$.
The set of meet-irreducible elements of $\cs$ is denoted $\M(\cs)$ or simply $\M$ 
when 
clear from the context.
The whole closure system can be recovered by taking the intersections of every 
combinations of meet-irreducible elements.
For a given closed set $C$, we put $\M(C) = \{M \in \M \mid C \subseteq M\}$.
We have $C = \bigcap \M(C)$.

Closure systems are closely related to closure operators.
A mapping $\cl \colon \pow{\U} \to \pow{\U}$ is a \emph{closure operator} if for 
every 
$X, Y \subseteq \U$, $X \subseteq \cl(X)$ ($\cl$ is extensive), $X \subseteq Y$ 
implies 
that $\cl(X) \subseteq \cl(Y)$ ($\cl$ is monotone) and $\cl(\cl(X)) = \cl(X)$ 
($\cl$ is 
idempotent).
The family $\cs = \{\cl(X) \mid X \subseteq \U \} = \{X \subseteq \U \mid \cl(X) = 
X\}$ 
is a closure system. 
Similarly, every closure system $\cs$ induces a closure operator $\cl$ defined by 
$\cl(X) 
= \bigcap \{C \in \cs \mid X \subseteq C\}$ for every $X \subseteq \U$.
Note that since $\cs$ is closed by intersection, we also have that $\cl(X) = 
\min(\{C \in 
\cs \mid X \subseteq C \})$.
Thus, the correspondence between closure operators and closure systems is 
one-to-one.

Let $\cs$ be a closure system over $\U$ with associated closure operator $\cl$.
We say that $\cs$ is \emph{standard} if for every $u \in \U$, $\cl(u) \setminus 
\{u\}$ is 
closed.
In particular, $\emptyset$ is closed.
In this paper, all the closure systems are considered standard, a common 
assumption 
\cite{adaricheva2013ordered, wild2017joy}.

A standard closure system $\cs$ over $\U$ is \emph{Boolean} if $\cs = \pow{\U}$.
It is \emph{distributive} if $C_1 \cup C_2 \in \cs$ for every pair of closed sets 
$C_1, 
C_2$.
Let $\cs_1, \cs_2$ be two closure systems over disjoint $\U_1, \U_2$ (resp.).
The \emph{direct product} of $\cs_1$ and $\cs_2$ is defined by $\cs_1 \times \cs_ 
2 = \{ 
C_1 \cup C_2 \mid C_1 \in \cs_1, C_2 \in \cs_2\}$.

\paragraph{Implicational bases}
An \emph{implication} over $\U$ is an expression $A \imp B$ where $A$ and $B$ are 
subsets 
of $\U$.
In $A \imp B$, $A$ is the \emph{premise} and $B$ the \emph{conclusion}.
An \emph{implicational base} $\is$ over $\U$ is a family of implications (over 
$\U$).
The size $\card{\is}$ of $\is$ is the number of implications it contains.
A subset $C$ of $\U$ \emph{satisfies} or \emph{models} an implicational base $\is$ 
if for 
every $A \imp B \in \is$, $A \subseteq C$ implies that $B \subseteq C$.
It is known \cite{bertet2018lattices, wild2017joy} that the family $\cs = \{C 
\subseteq 
\U \mid C \text{ satisfies } \is \}$ is a closure system.
Its associated closure operator $\cl$ can be computed with the closure procedure 
(or the 
forward chaining) \cite{ganter2012formal}.
For a given $X \subseteq \U$, this procedure starts from $X$ and constructs a 
sequence $X 
= X_0 
\subseteq \dots \subseteq X_k = \cl(X)$ of subsets of 
$\U$ such that for every $1 \leq i \leq k$, $X_i = X_{i - 1} \cup \bigcup \{B \mid 
\exists A \imp B \in \is \text{ such that } A \subseteq X_{i - 1} \}$.
The routine stops when $X_{i - 1} = X_i$. 

Dually, every closure system $\cs$ can be represented by at least one 
implicational base $\is$ \cite{wild2017joy}.
An implication $A \imp B$ holds in a closure system $\cs$ if all the closed sets 
of $\cs$ 
are models of $A \imp B$.
Equivalently, $A \imp B$ holds in $\cs$ if $B \subseteq \cl(A)$.
Two implicational bases are \emph{equivalent} if they represent the same closure 
system.
In particular, an implicational base $\is$ is equivalent to its 
\emph{unit-expansion} 
$\is_u = \{A \imp b \mid A \imp B \in \is, b \in B\}$.
We will interchangeably use an implicational base or its unit-expansion.

\begin{remark}
	As we restrict our attention to standard closure systems, we consider that an 
	implicational base $\is$ has no implications of the form $\emptyset \imp B$ 
	for some 
	$B 
	\subseteq \U$.
\end{remark}

Let $\is$ be an implicational base over $\U$.
The \emph{restriction} of $\is$ to a subset $\U'$ of $\U$ is the implicational 
base 
$\is[\U'] = \{A \imp b \in \is \mid A \cup \{b\} \subseteq \U' \}$.
Then, $\is[\U']$ is a \emph{sub-base} of $\is$.
Let $\U_1, \U_2$ be a non-trivial (full) bipartition of $\U$, that is $\U_1 \cup 
\U_2 = 
\U$, $\U_1 \cap \U_2 = \emptyset$ and $\U_1 \neq \emptyset$, $\U_2 \neq \emptyset$.
An implicational base is \emph{bipartite} (\wrt $\U_1, \U_2$) if every implication 
$A 
\imp 
B$ satisfies $A \subseteq \U_1$ and $B \subseteq \U_2$ or vice-versa.
We write $\is[\U_1, \U_2]$ to denote a bipartite implicational base.
A \emph{path} in $\is$ is a sequence $v_1, \dots, v_k$ of elements of $\U$ such 
that for 
every $1 \leq i < k$ there exists an implication $A_i \imp B_i$ with $v_i \in A_i$ 
and 
$v_{i + 1} \in B_i$.
The path is a \emph{cycle} when $v_1 = v_k$.
An implicational base without cycles is called \emph{acyclic}.
A closure system which admits an acyclic implicational base is an \emph{acyclic 
	convex geometry} \cite{edelman1985theory}. 
Acyclic convex geometries are also known as $G$-geometries \cite{wild1994theory} 
or poset 
type convex geometries \cite{adaricheva2014implicational}.
The term acyclic comes from Horn logic and acyclic Horn formulas 
\cite{hammer1995quasi, 
	zanuttini2015proprietes}.

Directed hypergraphs \cite{ausiello1986minimal} are a convenient 
representation for (unit-expansions of) implicational bases.
A \emph{directed hypergraph} $\cc{D}$ (over $\U$) is a pair $(\U, \A)$ where $\A$ 
is a 
set of hyperarcs.
A hyperarc is a pair $(A, b)$ where $A \cup \{b\} \subseteq \U$, $A$ is the 
\emph{body} 
and $b$ the \emph{head} of the hyperarc.
A hyperarc can be used to model an implication $A \imp b$ in the unit-expansion of 
an 
implicational base $\is$.

\begin{example} \label{ex:intro}
	Let $\U = \{1, 2, 3, 4\}$ and $\is = \{12 \imp 3, 23 \imp 4, 4 \imp 1 \}$. 
	The sequence $1, 3, 4$ is a cycle in $\is$.
	We represent $\is$ and its associated closure system $\cs$ in Figure 
	\ref{fig:intro}.
	The meet-irreducible elements of $\cs$ are $2, 14, 13$ and $134$.
	
	\begin{figure}[h!]
		\centering 
		\includegraphics[scale=1.0, page=1]{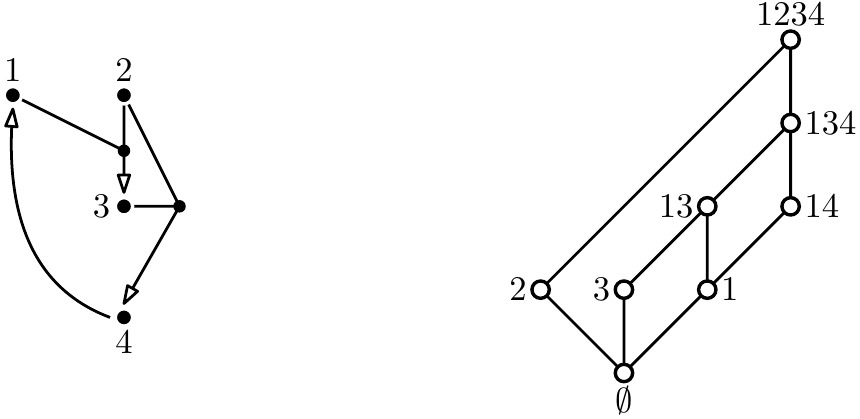}%
		\caption{On the left, the (associated directed hypergraph of the) 
		implicational 
		base 
			$\is$ in Example \ref{ex:intro}. On the right, its associated closure 
			system 
			$\cs$.}
		\label{fig:intro}
	\end{figure}
	
\end{example}

\paragraph{Enumeration complexity} We conclude with a brief reminder on 
enumeration 
algorithms \cite{johnson1988generating}.
Let $\ctt{A}$ be an algorithm with input $x$ of size $n$ and output a set of 
solutions 
$R(x)$ with $m$ elements.
In our case, each solution in $R(x)$ has size $\poly(n)$.
We say that $\ctt{A}$ is running in \emph{output-polynomial} time if its execution 
time 
is bounded by $\poly(m + n)$.
If the execution time of $\ctt{A}$ is instead bounded by $(n + m)^{\log(n + m)}$, 
\ctt{A} is said to run in \emph{output-quasipolynomial} time.

\section{Splits and hierarchical decomposition of implicational bases}
\label{sec:trad:split}

Inspired by \cite{dasgupta2016cost, libkin1993direct}, we define the \emph{split} 
operation for an implicational base $\is$ over $\U$.
A split is a bipartition ($\U_1, \U_2)$ of the groundset $\U$ which 
\emph{completely} 
partitions the implications of $\is$ in three sub-bases:
\begin{itemize}
	\item $\is[\U_1]$: the implications of $\is$ fully contained in $\U_1$,
	\item $\is[\U_2]$: the implications of $\is$ fully contained in $\U_2$,
	\item $\is[\U_1, \U_2]$: the implications of $\is$ whose premises are 
	included in 
	$\U_1$ and their conclusions in $\U_2$, or vice-versa.
\end{itemize}
This partitioning operation can be conducted recursively and leads to a 
\emph{hierarchical decomposition} (\emph{H-decomposition}) of $\is$, represented 
by a 
full rooted binary tree.
The root of the tree is labelled by $\is[\U_1, \U_2]$, its left-child corresponds 
to a 
decomposition of $\is[\U_1]$, its right-child to a decomposition of $\is[\U_2]$.
This tree is called a \emph{$\is$-tree}.
We illustrate the structure of a $\is$-tree in Figure \ref{fig:trad:over-tree}.

\begin{figure}[ht!]
	\centering
	\includegraphics[scale=1.0, page=1]{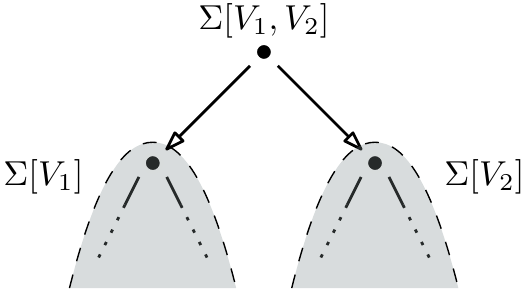}%
	\caption{A bipartition of $\is$ by a split.}
	\label{fig:trad:over-tree}
\end{figure}

We characterize the implicational bases having a hierarchical decomposition into 
trivial bases.
Moreover, we give a polynomial time and space algorithm, \ctt{BuildTree}, which 
takes an implicational base $\is$ as input, and outputs a $\is$-tree if it exists. 
Afterwards, we relax the requirement of the H-decomposition into trivial bases to 
H-factors, which are indecomposable sub-bases of $\is$.

Finally, we consider the decomposition of $\cs$, when a split $(\U_1, 
\U_2)$ of $\is$ is given. 
We show that $\cs$ is obtained by combining closed sets of $\cs_1$, the closure 
system of 
$\is[\U_1]$, with closed sets of $\cs_2$, the closure system of $\is[\U_2]$.
The way $\cs_1$ and $\cs_2$ are combined depends on the implications in $\is[\U_1, 
\U_2]$.

\subsection{Split operation}

Our first step is to define the split operation. 

\begin{definition} \label{def:trad:split}
	Let $\is$ be an implicational base over $\U$.
	A \emph{split} of $\is$ is a non-trivial bipartition $(\U_1, \U_2)$ of $\U$ 
	such that 
	for 
	every $A \imp b \in \is$, $A \subseteq \U_1$ or $A \subseteq \U_2$.
\end{definition}

A split $(\U_1, \U_2)$ induces three sub-bases $\is[\U_1]$, $\is[\U_2]$ and a 
bipartite base $\is[\U_1,\U_2]$. 
Moreover, every implication of $\is$ belongs to exactly one of $\is[\U_1]$, 
$\is[\U_2]$ 
or $\is[\U_1, \U_2]$ (recall that $\is$ has no implications $\emptyset \imp b$).
Intuitively, the split shows that $\is$ is fully described by two smaller 
distincts 
bases $\is[\U_1]$ and $\is[\U_2]$ acting on each other through the bipartite 
implicational base $\is[\U_1,\U_2]$. 

\begin{example} \label{ex:trad:split}
	Let $\U = \{1, 2, 3, 4, 5, 6, 7\}$ and consider the implicational base $\is$ 
	with 
	implications $12 \imp 3, 3 \imp 1, 56 \imp 2, 23 \imp 7, 45 \imp 6$ and $5 
	\imp 7$.
	Figure \ref{fig:trad:split-is} represents $\is$.
	
	\begin{figure}[ht!]
		\centering 
		\includegraphics[scale=1.0, page=1]{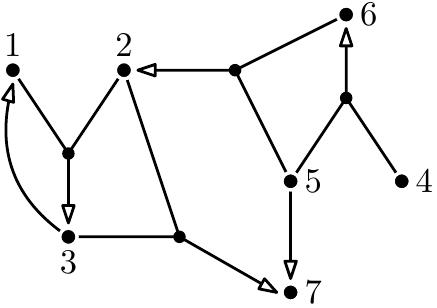}
		\caption{The implicational base of Example \ref{ex:trad:split}.}
		\label{fig:trad:split-is}
	\end{figure}
	
	In Figure \ref{fig:trad:split-bip} we consider two possible bipartitions of 
	$\U$.
	The bipartition illustrated on the left separates $\U$ in two sets $\U_1 = 
	\{1, 3\}$ 
	and 
	$\U_2 = \{2, 4, 5, 6, 7\}$. 
	It is not a split since the premises of $12 \imp 3$ and $23 \imp 7$ intersect 
	both 
	$\U_1$ 
	and $\U_2$.
	The bipartition on the right puts $\U_1=\{1, 2, 3\}$ and $\U_2=\{4, 5, 6, 7\}$.
	It is a split with $\is[\U_1] = \{12 \imp 3, 3 \imp 1\}$, $\is[\U_2] = \{45 
	\imp 6, 5 
	\imp 7\}$, and $\is[\U_1,\U_2] = \{56 \imp 2, 23 \imp 7\}$. 
	\begin{figure}[ht!]
		\centering 
		\includegraphics[scale=0.9, page=2]{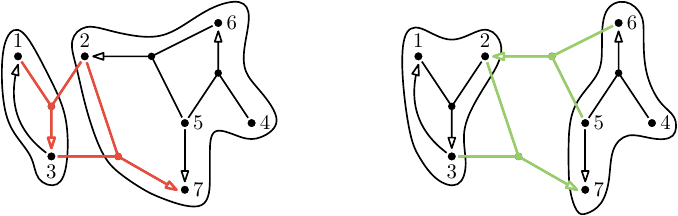}
		\caption{Two bipartitions of $\U$, the left one is not a split of $\is$, 
		the 
			right 
			one is.}
		\label{fig:trad:split-bip}
	\end{figure}
\end{example}

Before giving a characterization of implicational bases having a split, we make 
two 
observations.
First, $\is$ is empty or contains only implications of the form $a \imp b$. 
In this case, every non-trivial bipartition of $\U$---every cut of the associated 
directed (hyper)graph---is a split. 
In fact, an implication of the form $a \imp b$ always satisfies the condition of 
Definition \ref{def:trad:split}.
Thus, these implications have no impact on the existence of a split.
Second, there may be implicational bases where no bipartition corresponds to 
a split, as shown by the next example.

\begin{example} \label{ex:trad:fail-split}
	Consider $\U = \{1, 2, 3\}$ and the implicational base $\is = \{12 \imp 3, 
	\allowbreak 13 \imp 2\}$.
	Here, none of the three possible bipartitions is a split:
	\begin{itemize}
		\item $\U_1 = \{1, 2\}$ and $\U_2 = \{3\}$ fails to separate the 
		implication $13 
		\imp 
		2$;
		\item $\U_1 = \{1, 3\}$, $\U_2 = \{2\}$ omits the implication $12 \imp 3$; 
		and
		\item $\U_1 = \{2, 3\}$, $\U_2 = \{1\}$ breaks the two implications of 
		$\is$.
	\end{itemize}
\end{example}

In the following, we show that the implicational base's connectivity is important 
for 
the notion of a split. 
Let $\is$ be an implicational base over $\U$.
A \emph{premise-path} in $\is$ is a sequence $v_1, \dots, v_k$ of (distinct) 
elements of 
$\U$ such that for every $1 \leq i < k$ there exists an implication $A_i \imp b_i$ 
in 
$\is$ such that $\{v_i, v_{i+1}\} \subseteq A_i$.
Two vertices $u, v\in \U$ are said to be 
\emph{premise-connected} in $\is$ if there exists a premise-path from $u$ to $v$.
We say that $\is$ is \emph{premise-connected} when every pair of vertices in $\U$ 
is 
premise-connected. 
A subset $C$ of $\U$ is a \emph{premise-connected component} of $\is$ if there 
exists a 
premise-path between each pair of vertices of $C$, and if $C$ is inclusion-wise 
maximal 
for this property.
A singleton premise-connected component of $\is$ is \emph{trivial}.

\begin{example} 
	Consider the implicational base $\is$ given in Example 
	\ref{ex:trad:split}.
	For instance, $6, 5, 4$ is a premise-path and hence $4$ and 
	$6$ are premise-connected.
	Here $\is$ is not premise-connected as there is no premise-path 
	between $2$ and $6$.
	The premise-connected components of $\is$ are $\{1, 2, 3\}$, $\{4, 5, 6\}$ and 
	$\{7\}$ being trivial.
\end{example}

Using premise-connectivity, we are now in position to identify whether a given 
implicational base admits a split or not.

\begin{proposition} \label{prop:trad:premise-connectivity}
	An implicational base $\is$ over $\U$ has a split if and only if it is not 
	premise-connected.
\end{proposition}

\begin{proof}
	We begin with the only if part. 
	Suppose that $\is$ has a split $(\U_1, \U_2)$, and let $u \in \U_1$ and $v \in 
	\U_2$.
	Since a split is a non-trivial bipartition of $\U$, such $u$ and $v$ must 
	exist.
	Now let us assume for contradiction there exists a premise-path $u = v_1, 
	\dots, v_k 
	= v$ 
	for some $k \in \cb{N}$.
	Such a premise-path exists if there is some $j$ with $1 \leq j \leq k$ such 
	that $A_j 
	\imp b_j$ is an implication of $\is$, $A_j \cap \U_1 \neq \emptyset$ and $A_j 
	\cap 
	\U_2 \neq \emptyset$.
	However, the implication $A_j \imp b_j$ does not satisfy Definition 
	\ref{def:trad:split}.
	This contradicts the assumption that $(\U_1, \U_2)$ is a split of $\is$. 
	Hence, $u, v$ cannot be premise-connected and $\is$ is not premise-connected 
	either.
	
	We move to the if part.
	Suppose that $\is$ is not premise-connected and let $C$ be a premise-connected 
	component 
	of $\is$. 
	We show that $(C, \U \setminus C)$ is a split of $\is$. 
	Let $A \imp b$ be an implication in $\is$.
	If $A \subseteq C$ or $A$ is a singleton element, it is clear that it 
	satisfies 
	Definition \ref{def:trad:split}.
	Assume that $A \nsubseteq C$ and that $A$ is not a singleton element.
	Recall that no implication of the form $\emptyset \imp b$ lies in $\is$.
	Let $u, v$ be distinct elements in $A$ and assume for contradiction $u \in C$ 
	and $v 
	\notin C$.
	Clearly, $u, v$ is a premise path between $u$ and $v$.
	Let $w$ be any element of $C$.
	Since $u \in C$, $u$ and $w$ are premise connected.
	Consider any premise-path from $w$ to $u$ and append $v$ to its end.
	The new path is a premise-path connecting $w$ and $v$.
	Hence, $C \cup \{v\}$ is premise-connected, a contradiction with the fact that 
	$C$ is 
	maximal.
	We deduce that $A \nsubseteq C$ implies that $A \cap C = \emptyset$.
	So $(C,\U\setminus C)$ is indeed a split of $\is$.
\end{proof} 

It is important to note that premise-connectivity is not inherited. 
That is, a sub-base induced by a premise-connected component needs not be 
premise-connected in general. 

\begin{example}
	Consider the implicational base of Example \ref{ex:trad:split} with the split 
	$ \U_1 
	= \{1, 2, 3\}$, $\U_2 = \{4, 5, 6, 7\}$.
	The elements $5$ and $6$ are premise-connected in $\is$ but not in $\is[\U_2] 
	= \{5 
	\imp 
	7, \allowbreak 45 \imp 6\}$.
	This happens because the implication $56 \imp 2$ is in $\is[\U_1, \U_2]$.
\end{example}

Henceforth, premise-connected components of an implicational base may be further 
decomposed.
Consequently, the split operation can be conducted in a recursive manner, leading 
to a 
hierarchical decomposition of implicational bases, up to trivial cases.

\subsection{The decomposition tree of an implicational base}

Based on the split operation, we define a hierarchical decomposition of an 
implicational 
base $\is$.
We call it a \emph{H-decomposition} of $\is$.
The strategy is to recursively split $\is$ into smaller implicational bases until 
we 
reach trivial cases.
This recursive decomposition can be conveniently represented by a full rooted 
binary 
tree $T$ (full means that each node has precisely two children). 
An interior node of the tree corresponds to a split $(\U_1,\U_2)$ of\ $\is$ whose 
children 
are H-decompositions of $\is[\U_1]$ and $\is[\U_2]$.
The leaves of the tree represent the ground set $\U$.
Since the splits $(\U_1,\U_2)$ and $(\U_2,\U_1)$ are equivalent, the children of a 
node 
are unordered.

\begin{definition}[$\is$-tree and H-decomposition] \label{def:trad:is-tree}
	Let $\is$ be an implicational base over $\U$ and $T$ be a full rooted binary 
	tree. 
	Then $(T, \lambda)$ is a $\is$-tree of $\is$ if there exists a labelling map 
	$\lambda 
	\colon T \rightarrow \U \cup 2^{\is}$ 
	satisfying the following conditions:
	\begin{enumerate}
		\item $\lambda(t)$ equals $v$ for some $v \in \U$ if $t$ is a 
		leaf of $T$;
		
		\item $\lambda(t) \subseteq \is$ if $t$ is an interior node (possibly 
		$\lambda(t) = \emptyset$);
		
		\item for every $A \imp b \in \lambda(t)$, elements of $A$ are labels of 
		leaves 
		in the subtree of one child of $t$ and $b$ is the label of a leaf in the 
		subtree 
		of the other child. 
		
		\item the set $\{\lambda(t) \mid t \in T\}$ is a full partition of \ $\U 
		\cup 
		\is$ and may contain the empty set. 
	\end{enumerate}
	If such labelling exists, we say that $\is$ is \emph{hierarchically 
	decomposable}
	(\emph{H-decomposable} for short), and \emph{H-indecomposable} otherwise.
\end{definition}

In the particular case where $\U = \emptyset$, we must have that $\is = \emptyset$.
If it happens, we say for convenience that $\is$ is trivially H-decomposable and 
that its 
$\is$-tree is empty.

\begin{example}
	The implicational base $\is$ from Example \ref{ex:trad:split} is 
	H-decomposable. 
	In Figure \ref{fig:trad:is-tree}, we represent a possible $\is$-tree for 
	$\is$. 
	
	\begin{figure}[ht!]	
		\centering 
		\includegraphics[scale=1.0, page=1]{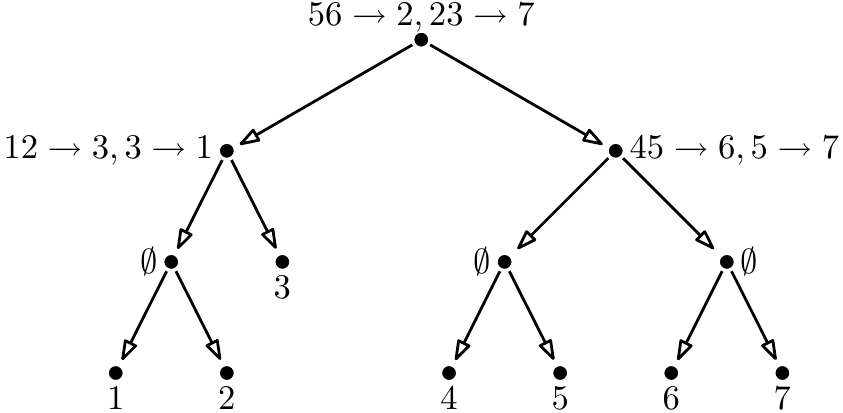}%
		\caption{An $\is$-tree for the implicational base of Example 
		\ref{ex:trad:split}.}
		\label{fig:trad:is-tree}
	\end{figure}
	
\end{example}

There are cases where a H-decomposition can be computed easily.
For instance, if $\is$ is empty, every full rooted binary tree whose 
leaves are labelled by a permutation of $\U$ and every interior node by 
$\emptyset$ is 
a $\is$-tree.
The case where $\is$ only contains implications of the form $a \imp b$ for some 
$a, b \in 
\U$ behaves similarly, except that the interior nodes of the tree contain the 
implications of $\is$.
However, there are also some implicational bases that cannot be H-decomposed, for 
example when they admit no split at all.

Thus, our objective is to characterize H-decomposable implicational bases and 
devise a polynomial-time algorithm to build decomposition trees whenever possible.
We first need two preparatory propositions.

\begin{proposition} \label{prop:trad:H-decomposition}
	A H-decomposable implicational base $\is$ is not premise-connected.
\end{proposition}

\begin{proof}
Suppose that $\is$ is H-decomposable, and let $(T, \lambda)$ be a $\is$-tree with root $r$. 
Let $(\U_1, \U_2)$ be the split of $\U$ corresponding to $r$, \ie $\U_1$ 
corresponds 
to the leaves of the left subtree of $r$ and $\U_2$ to those of the right 
subtree. 
Then, according to Proposition \ref{prop:trad:premise-connectivity}, $\is$ is 
not 
premise-connected.
\end{proof} 

Remark that the converse of Proposition \ref{prop:trad:H-decomposition} does not 
hold in 
general.
We exhibit a counter-example.
The main idea is to hide a premise-connected implicational base into a sub-base of 
a non 
premise-connected one.

\begin{example}
	Let $\U = \{1, 2, 3, 4\}$ and $\is = \{12 \imp 3, 13 \imp 2, 23 \imp 4 \}$.
	The implicational base $\is$ has a unique split, $\U_1 = \{1, 2, 3\}$ and 
	$\U_2 = 
	\{4\}$.
	Thus it is not premise-connected and any possible $\is$-tree must have the 
	split 
	$(\U_1, 
	\U_2)$ in the label of its root.
	After splitting, we are left with the sub-bases $\is[\U_2] = \emptyset$, 
	$\is[\U_1, 
	\U_2] 
	= \{23 \imp 4\}$ and $\is[\U_1] = \{12 \imp 3, 13 \imp 2\}$.
	Observe that $\is[\U_1]$ is exactly the implicational base of Example 
	\ref{ex:trad:fail-split}.
	Hence, it is premise-connected and using Proposition 
	\ref{prop:trad:H-decomposition}, 
	it cannot be H-decomposed.
	It follows that $\is$ admits no H-decomposition either.
\end{example}

Inspired by the previous example, we show that H-decomposability is hereditary, 
\ie if an 
implicational base $\is$ has a $\is$-tree then each of its sub-bases has a 
H-decomposition too.

\begin{proposition}\label{prop:trad:hereditary}
	Let $\is$ be an implicational base over $\U$ and let $X \subseteq \U$.
	Then $\is$ has a H-decomposition only if $\is[X]$ is H-decomposable.
\end{proposition}

\begin{proof}
	Let $\is$ be an implicational base over $\U$, $X \subseteq \U$, and 
	let $(T, \lambda)$ be a $\is$-tree.
	If $X = \emptyset$, then the result trivially holds.
	We construct a subtree not necessarily induced by $T$ which corresponds to a 
	$\is[X]$-tree. 
	We start from the root $r$ of $T$ and apply the following operation for each 
	interior 
	node $t$: if the sets of leaves of the left child and those of the right one 
	both 
	intersect $X$, keep $t$ with label $\lambda(t)= \lambda(t) \cap \is[X]$. 
	Otherwise, there is a child of $t$ whose set of leaves do not intersect $X$.
	In this case replace $t$ by the child whose set of leaves intersects $X$.
	In the resulting subtree, the leaves are labelled by the elements of $X$, and the internal nodes by the implications of $\is[X]$.
\end{proof}

The following theorem characterizes H-decomposability and gives the strategy of an 
algorithm computing a H-decomposition.

\begin{theorem} \label{thm:trad:H-decomposability}
	Let $\is$ be a non premise-connected implicational base and let $C$ be a 
	premise-connected component of $\is$.
	Then $\is$ is H-decomposable if and only if $\is[C]$ and $\is[\U \setminus C]$ 
	are 
	H-decomposable.
\end{theorem}

\begin{proof}
The only if part directly follows from Proposition \ref{prop:trad:hereditary}.
Let us show the if part.
Let $C$ be a premise-connected component of $\is$, $(T_1, \lambda_1)$ be a 
$\is[C]$-tree and $(T_2, \lambda_2)$ be a $\is[\U \setminus C]$-tree. 
We consider a new tree $(T, \lambda)$ such that $T$ has root $r$ with left 
subtree 
$T_1$ and right subtree $T_2$.
As for $\lambda$, we put $\lambda(t_1) = \lambda_1(t_1)$ if $t_1 \in T_1$, 
$\lambda(t_2) 
= \lambda_2(t_2)$ if $t_2 \in T_2$ and $\lambda(r) = \is \setminus (\is[C] 
\cup 
\is[\U 
\setminus C])$. 
In other words, $\lambda(r)$ contains each implication whose premise is not fully 
contained 
in 
$C$ or $\U \setminus C$.
It is clear that conditions \emph{1, 2, 4} of Definition 
\ref{def:trad:is-tree} are fulfilled for $(T, \lambda)$ as they are for 
$(T_1, \lambda_1)$, $(T_2, \lambda_2)$ and $C \cup \U \setminus C = \U$. 
Hence, we have to check \emph{3}. 
Let $A \imp b$ be an implication in $\lambda(v)$.
If $A \cap C \neq \emptyset$, then $A \subseteq C$ since $C$ is a 
premise-connected 
component of $\is$.
As $A \imp b$ is not an implication of $\is[C]$, it follows that $b \in \U 
\setminus 
C$.
Dually, if $A \cap C = \emptyset$, then $b \in C$ since $A \imp b$ is not in 
$\is[\U 
\setminus C]$.
Consequently, condition \emph{3} is satisfied and $(T, \lambda)$ is a 
$\is$-tree 
as 
required.
\end{proof}

Theorem \ref{thm:trad:H-decomposability} suggests a recursive algorithm which 
returns a 
$\is$-tree for an implicational base $\is$ if it is H-decomposable. 
If $\U = \emptyset$, we simply output $\emptyset$.
If $\U$ is a singleton element $v$, we output a leaf with label $v$.
Otherwise, we compute a premise-connected component $C$ of $\is$ if $\is$ is not 
premise-connected.
We label the corresponding node by the implications of $\is[C, \U \setminus C]$, 
and we recursively call the algorithm on $\is[C]$ and $\is[\U \setminus C]$. 
This strategy is formalized in Algorithm \ref{alg:trad:build-tree}, whose 
correctness and 
complexity are studied in Theorem \ref{thm:trad:build-tree}. 

\begin{algorithm}
	\KwIn{An implicational base $\is$ over $\U$}
	\KwOut{A $\is$-tree, if it exists, \csf{FAIL} otherwise}
	
	\If{$\U = \emptyset$}{
		return $\emptyset$ \;
	}
	
	\If{$\U$ has one vertex $v$}{
		create a new leaf $r$ with appropriated $\lambda(r)$\; 
		return $r$ \;
	}\Else{
		compute a premise-connected component $C$ of $\is$ \;
		\If{$\card{C} =\card{\U}$}{
			stop and return \csf{FAIL} \;
		}
		\Else{
			let $r$ be a new node with $\lambda(r) =  \is \setminus (\is[C] \cup 
			\is[\U 
			\setminus C])$ \;
			$\csf{left}(r)= $ \ctt{BuildTree}$(\is[C])$ \;
			$\csf{right}(r)  = $ \ctt{BuildTree}$(\is[\U \setminus C])$ 
			\;			
			return $r$ \;
		}
	}
	
	\caption{\ctt{BuildTree}.}
	\label{alg:trad:build-tree}
\end{algorithm}

\begin{theorem} \label{thm:trad:build-tree}
Given an implicational base $\is$ over $\U$, the algorithm \nf{\ctt{BuildTree}} computes 
a $\is$-tree if it exists, in $O(\card{\U}^2 \times \card{\is} \times 
\alpha(\card{\is} \times \card{\U}, \card{\U}))$ time and $O(\card{\is} \times 
\card{\U})$ space, where $\alpha(\cdot, \cdot)$ is the inverse of the Ackermann function.
\end{theorem}

\begin{proof}
	First, we show by induction on $\card{\U}$ that the algorithm returns a 
	$\is$-tree if 
	and 
	only if $\is$ is H-decomposable. 
	Clearly if $\U = \emptyset$, the algorithms returns $\emptyset$.
	In the case where $\U$ is reduced to a vertex $v$, the algorithm returns a 
	$\is$-tree corresponding to a leaf with label $v$.
	
	Now, assume that the algorithm is correct for implicational bases with 
	$\card{\U} < 
	n$, 
	$n \in \cb{N}$, and consider a base $\is$ over $\U$ with $\card{\U} = n$.
	Suppose $\is$ is H-decomposable. 
	By Proposition \ref{prop:trad:premise-connectivity}, $\is$ is not 
	premise-connected.
	Let $C$ be a premise-connected component of $\is$.
	Inductively, the algorithm is correct for $\is[C]$ and $\is[\U \setminus C]$ 
	since $ 
	1 \leq \card{C} < n$.
	From Theorem \ref{thm:trad:H-decomposability}, we have that both $\is[C]$ and 
	$\is[\U 
	\setminus C]$ are H-decomposable.
	By induction, the algorithm computes a $\is[C]$-tree $(T_1, \lambda_1)$ and a 
	$\is[\U \setminus C]$-tree $(T_2, \lambda_2)$. 
	Hence, the algorithm returns a labelled tree $(T, \lambda)$ with root $r$ 
	whose label 
	is $\lambda(r) = \is \setminus (\is[C] \cup \is[\U \setminus C])$ and children 
	$T_1$ and $T_2$.
	This tree satisfies all conditions to be a $\is$-tree.
	Thus, the algorithm computes a $\is$-tree for every H-decomposable 
	implicational 
	base. 
	
	Now suppose $\is$ is not H-decomposable. We have two cases:
	\begin{enumerate}
		\item $\is$ is premise-connected and the algorithm returns \csf{FAIL} in 
		Line \textbf{9}.
		\item $\is$ is not premise-connected.
		The algorithm chooses a premise-connected component $C$ with $1 \leq 
		\card{C} < n$. 
		By Theorem \ref{thm:trad:H-decomposability}, either $\is[C]$ or $\is[\U 
		\setminus C]$ 
		is H-indecomposable. 
		Thus, by induction, the algorithm will return \csf{FAIL} for the input 
		$\is[C]$ 
		or 
		$\is[\U \setminus C]$ in lines \textbf{11}-\textbf{14}. 
		Since the algorithm stops, the output of the algorithm is \csf{FAIL}.
	\end{enumerate}
	Hence, the algorithm fails if the input $\is$ is H-indecomposable.
	We conclude that the algorithm returns a $\is$-tree if and only if the input 
	$\is$ is H-decomposable.
	
	Finally, we show that the total time and space complexity of the algorithm are 
	polynomial. 
	The space required for the algorithm is bounded by the size of the 
	implicational base 
	$\is$, the ground set $\U$ and the size of the $\is$-tree.
	As the size of the $\is$-tree is bounded by $O(\card{\is} \times \card{\U})$, 
	the 
	overall 
	space is bounded by $O(\card{\is} \times \card{\U})$.
	
	The time complexity is bounded by the sum of the costs of all nodes (or calls) 
	of 
	the search tree. 
	The number of calls is bounded by $O(\card{\U})$, the size of the search tree.
	The cost of a call is dominated by the computation of a premise-connected 
	component of the input $\is$.
	For this, we use union-find data structure of \cite{tarjan1984worst}, which 
	runs in 
	almost 
	linear time, \ie $O(\card{\is} \times \card{\U} \times \alpha(\card{\is} 
	\times 
	\card{\U}, \card{\U}))$ where $\alpha(.,.)$ is the inverse Ackermann function. 
	The almost linear comes from the fact that $\alpha(\card{\U}) \leq 4$ for every
	practical implicational base (see \cite{tarjan1984worst}). 
	Thus, the total time complexity is $O(\card{\U}^2 \times \card{\is} \times 
	\alpha(\card{\is} \times \card{\U}, \card{\U}))$.
\end{proof}

It is worth noticing, that the $\is$-tree we obtain by the end of Algorithm 
\ref{alg:trad:build-tree} depends on the choice of a premise-connected component 
in line \textbf{7}.
As shown by the following example, the structure of the resulting $\is$-tree is 
impacted by this choice.

\begin{example} \label{ex:trad:multiple-tree}
Let $\U = \{1, 2, 3, 4, 5, 6, 7, 8\}$ and let $\is$ be the implicational base 
$\{12 \imp 3, 23 \imp 4, \allowbreak 34 \imp 5, 56 \imp 7, 67 \imp 
8\}$.
For convenience, we represent $\is$ in Figure \ref{fig:trad:multiple-tree-DH}.

\begin{figure}[ht!]
	\centering 
	\includegraphics[scale=1.0, page=1]{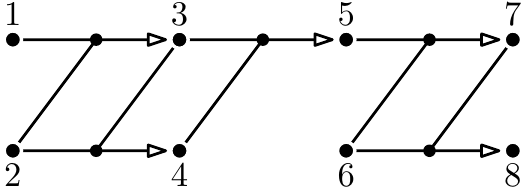}%
	\caption{The implicational base of Example \ref{ex:trad:multiple-tree}.}
	\label{fig:trad:multiple-tree-DH}
\end{figure}

The premise-connected components of $\is$ are $\{1, 2, 3, 4\}$, $\{5, 6, 7\}$ 
and 
$\{8\}$.
Thus, we can devise at least three distinct $\is$-trees for $\is$.
In Figure \ref{fig:trad:multiple-tree-trees}, we give two of them.
On the one hand, the $\is$-tree on the left balances the size of labels of its 
interior nodes.
On the other hand, the $\is$-tree on the right is balanced.

\begin{figure}[ht!]
	\centering 
	\includegraphics[scale=0.8, page=2]{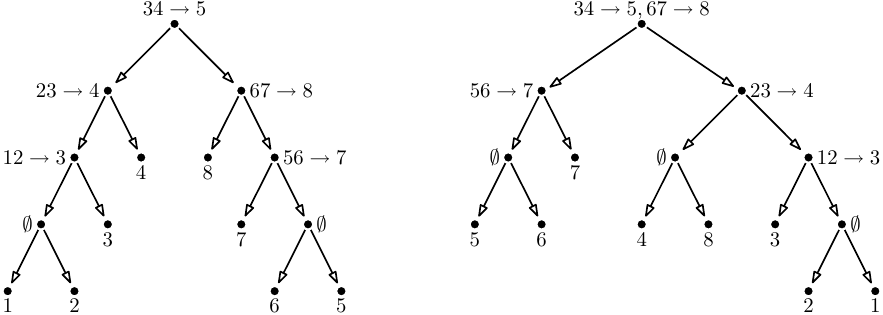}%
	\caption{Two $\is$-trees for the implicational base of Example 
		\ref{ex:trad:multiple-tree}.}
	\label{fig:trad:multiple-tree-trees}
\end{figure}
	
\end{example}

Following the previous example, a natural question arises: are all 
\textit{$\is$-trees 
	equivalently interesting?}
In particular, a balanced $\is$-tree is a good candidate as the 
balancing is a common desirable property for decomposition trees to obtain 
efficient 
algorithms.
This question, which uniquely depends on the syntax of the implicational base, is 
left 
open for further research.

\subsection{Extension of the H-decomposition}

As seen before, there are implicational bases that cannot have a split and thus that cannot have a H-decomposition into trivial sub-bases.
Such implicational bases are premise-connected, and will be called 
\textit{irreducible 
	H-factors} (H-factors for short).  
Now  we describe a slight modification of Algorithm \ref{alg:trad:build-tree} to 
obtain a 
H-decomposition of implicational bases into H-factors. 
Instead of returning \csf{FAIL} at line \textbf{9} in Algorithm \ctt{BuildTree}, 
we replace it by the following:
\begin{center}
	\textbf{9'} create a new leaf $r$ with $\lambda(r) = \is$ and return $r$;
\end{center}
Algorithm \ref{alg:trad:build-tree-2}, called \ctt{H-BuildTree}, is the updated version of \ctt{BuildTree} with this modification.

\begin{algorithm}
	\KwIn{An implicational base $\is$ over $\U$}
	\KwOut{A $\is$-tree with $H$-factors}
	
	\If{$\U = \emptyset$}{
		return $\emptyset$ \;
	}
	
	\If{$\U$ has one vertex $v$}{
		create a new leaf $r$ with appropriated $\lambda(r)$\; 
		return $r$ \;
	}\Else{
		compute a premise-connected component $C$ of $\is$ \;
		\If{$\card{C} =\card{\U}$}{
			create a new leaf $r$ with $\lambda(r) = \is$ and return $r$\;
		}
		\Else{
			let $r$ be a new node with $\lambda(r) =  \is \setminus (\is[C] \cup 
			\is[\U \setminus C])$\;
			$\csf{left}(r)= $ \ctt{BuildTree}$(\is[C])$\;
			$\csf{right}(r)  = $ \ctt{BuildTree}$(\is[\U \setminus C])$ 
			\;			
			return $r$ \;
		}
	}
	
	\caption{\ctt{H-BuildTree}}.
	\label{alg:trad:build-tree-2}
\end{algorithm}

\begin{example} \label{ex:trad:h-factors}
	Consider $\U = \{1, 2, 3, 4, 5, 6\}$ and let $\is = \{45 \imp 1, 12 \imp 3, 23 
	\imp 1, 13 \imp 2, 3 \imp 6, \allowbreak 1 \imp 4 \}$.
	We represent $\is$ on the left of Figure \ref{fig:trad:h-factors}.
	Clearly, $\is$ is not premise-connected and its premise-connected components 
	are 
	$\{4, 5\}$, $\{1, 2, 3\}$ and $\{3\}$.
	On the right of Figure \ref{fig:trad:h-factors}, we present a H-decomposition 
	of 
	$\is$ 
	into H-factors.
	
	\begin{figure}[h!]
		\centering 
		\includegraphics[scale=1.0]{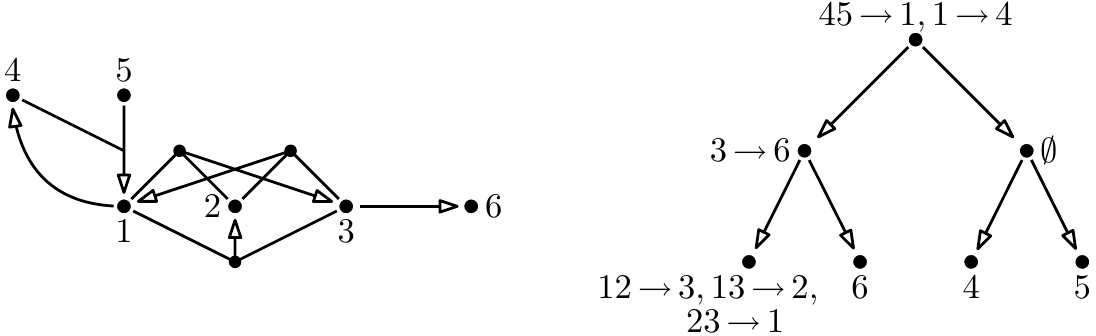}
		\caption{H-decomposition into H-factors.}
		\label{fig:trad:h-factors}
	\end{figure}
\end{example}

With this modification, each possible implicational base 
has now a H-decomposition  where leaves can be H-factors.
To conclude this subsection, we show that H-factors are independent of the choice 
of the 
$\is$-tree.

\begin{proposition} \label{prop:trad:leaves}
	Let $\is$ be an implicational base over $\U$ and let $(T_1, \lambda_1)$ and 
	$(T_2, 
	\lambda_2)$ be two $\is$-trees.
	Then, $T_1$ and $T_2$ have the same number of leaves and $\{\lambda_1(t_1) 
	\mid 
	t_1 \text{ is a leaf of } T_1\} = \{\lambda_2(t_2) \mid t_2 \text{ is a leaf 
	of } 
	T_2\}$.
\end{proposition}

\begin{proof}
	If $\is$ is H-decomposable or $(T_1, \lambda_1) = (T_2, \lambda_2)$, the 
	result is 
	clear 
	due to Theorem \ref{thm:trad:build-tree}.
	Assume that $\is$ is not H-decomposable and that the trees are different.
	Let $t_1$ be a leaf of $T_1$ such that $\lambda_1(t_1) = \is_H$ is a H-factor of 
	$\is$.
	Let $\U_H$ be the set of elements spanned by $\is_H$ and let $t_2$ be the 
	lowest node 
	of 
	$T_2$ such that $\is_{H} \subseteq \bigcup \{\lambda_2(t_2') \mid t_2 \text{ is 
	an ancestor of } t_2' \text{ in } T_2\}$.
	In other words, $t_2$ is the ancestor of all the elements in $\U_H$.
	If $t_2$ is not a leaf, there exists a split in the sub-base induced by $t_2$ 
	which 
	separates the elements of $\U_H$, a contradiction with $\is_H$ being a 
	H-factor of 
	$\is$ 
	in $(T_1, \lambda_1)$.
	Hence, $t_2$ is also a leaf, and $\lambda_2(t_2) = \is_H$ follows by applying 
	the 
	same 
	reasoning in $T_1$, which concludes the proof.
\end{proof}

\subsection{Splits and decomposition of a closure system}

Naturally, the H-decomposition of an implicational base $\is$ induces a 
decomposition of the closure system $\cs$ defined by $\is$.
We also call the decomposition of $\cs$ a H-decomposition. 
The H-decomposition of $\cs$ is obtained from the H-decomposition of $\is$, 
where the label of a node of its $\is$-tree is replaced by the closure system 
associated 
to the implicational base induced by its subtree. 
The closure systems in leaves are the irreducible H-factors of the input 
closure system. 
Figure \ref{fig:trad:h-factors-cs} illustrates the H-decomposition of the closure 
system 
associated to the H-decomposition of Example \ref{ex:trad:h-factors}. 
Recall that for a set system $\cc{S}$ over $\U$ and a subset $\U'$ of $\U$, 
$\cc{S} \colon \U'$ is the trace of $\cc{S}$ on $\U'$, that is $\cc{S} \colon \U' = \{S 
\cap \U' \mid S \in \cc{S}\}$.

\begin{figure}[h!]
	\centering 
	\includegraphics[scale=0.85, page=2]{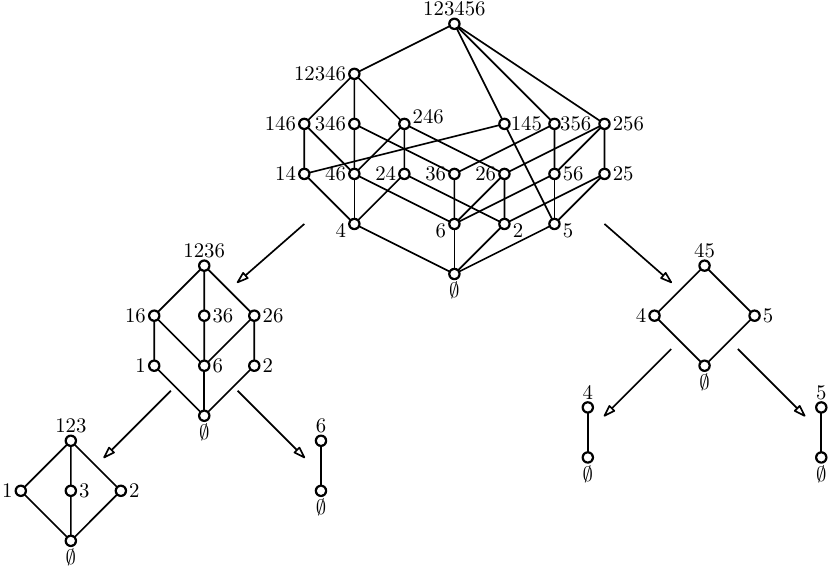}
	\caption{H-decomposition of the closure system corresponding to Example 
		\ref{ex:trad:h-factors}.}
	\label{fig:trad:h-factors-cs}
\end{figure}

\begin{theorem}  \label{thm:trad:H-decomposition-cs}
	Let $\is$ be an implicational base over $\U$ with closure system $\cs$, and 
	let 
	$(\U_1, 
	\U_2)$ be a split of $\is$.
	Let $\cs_1$ and $\cs_2$ be the closure systems associated to $\is[\U_1]$ and 
	$\is[\U_2]$ 
	(resp.).
	Then:
	\begin{enumerate}
		\item $C \in \cs$ implies that $C \cap \U_1 \in \cs_1$ and $C \cap \U_2 
		\in 
		\cs_2$.
		Hence, $\cs \subseteq \cs_1 \times \cs_2$;
		
		\item $\cs = \cs_1 \times \cs_2$ holds whenever $\is[\U_1, \U_2] = 
		\emptyset$ 
		(\ie 
		$\cs$ is the direct product of $\cs_1$ and $\cs_2$);
		
		\item if for every implication $A \imp b$ in $\is[\U_1, \U_2]$, we have $A 
		\subseteq 
		\U_1$, then $\cs \colon \U_1 = \cs_1$ and $\cs \colon \U_2 = \cs_2$; and
		
		\item dually, if $A \subseteq \U_2$ for every $A \imp b$ in 
		$\is[\U_1,\U_2]$, we 
		have
		$\cs \colon \U_1 = \cs_1$ and $\cs \colon \U_2 = \cs_2$.
		
	\end{enumerate}
\end{theorem}	

\begin{proof} 
	Consider a split $(\U_1, \U_2)$ of  $\is$, $\cs_1$ and $\cs_2$ the 
	closure systems corresponding to $\is[\U_1]$ and $\is[\U_2]$. 
	Their respective closure operators are $\cl_1$, $\cl_2$.
	We prove items \emph{1}, \emph{2} and \emph{3}.
	Statements \emph{3} and \emph{4} are similar.
	
	
	\emph{Item 1.}
	Let $C \in \cs$, $C_1 = C \cap \U_1$ and let $A \imp b$ be an implication of 
	$\is[\U_1]$.
	Suppose $A \subseteq C_1$ and $b \notin C_1$. 
	Then we also have $A \subseteq C$ and $b \notin C$ which contradicts $C \in 
	\cs$ as 
	$A 
	\imp b \in \is$.
	Thus $C_1\in \cs_1$.
	A similar reasoning applies to $\cs_2$, and $\cs \subseteq \cs_1 \times \cs_2$ 
	holds. 
	
	\emph{Item 2.}
	We readily have that $\cs \subseteq \cs_1 \times \cs_2$ by item \emph{(i)}. 
	For the other inclusion, let $C_1 \in \cs_1$ and $C_2 \in \cs_2$. 
	We show that $C_1 \cup C_2 \in \cs$. 
	Let $A \imp b$ be an implication of $\is$ with $A \subseteq C_1 \cup C_2$. 
	As $\is[\U_1,\U_2]$ is empty, $A \imp b$ is either an implication of 
	$\is[\U_1]$ or 
	$\is[\U_2]$.
	As $C_1, C_2$ are closed for $\is[\U_1]$, $\is[\U_2]$ (resp.), it follows that 
	$C_1 \cup C_2 \in \cs$.
	
	\emph{Item 3.}
	Let $C_1 \in \cs_1$.
	We show that $\cl(C_1)$ satisfies $\cl(C_1) \cap \U_1 = C_1$.
	We readily have that $C_1 \subseteq \cl(C_1) \cap \U_1$.
	Let $C_1 = X_0 \subset X_1 \subset \dots \subset X_k = \cl(C_1)$ be the 
	sequence of 
	sets 
	obtained by applying the forward chaining algorithm on $C_1$ with $\is$.
	We show by induction on $0 \leq i \leq k$ that $X_i \cap \U_1 = C_1$.
	For the initial case $X_0 = C_1$, the result is clear.
	Now assume that the results holds true for any $0 \leq i < k$ and consider 
	$X_{i + 
		1}$.
	Let $A \imp b$ be an implication such that $A \subseteq X_i$.
	Since $(\U_1, \U_2)$ is a split of $\is$, either $A \subseteq \U_1$ or $A 
	\subseteq 
	\U_2$.
	We have three cases
	\begin{enumerate}[label=(\arabic*)]
		\item $A \subseteq \U_2$.
		Then $A \imp b \in \is[\U_2]$ and $b \in \U_2$ so that $b \notin X_{i + 1} 
		\cap 
		\U_1$.
		
		\item $A \imp b$ is in $\is[\U_1]$.
		Then, $A \subseteq X_i \cap \U_1$ which equals $C_1$ by inductive 
		hypothesis.
		Since $C_1$ models $\is[\U_1]$ we have that $b \in X_i \cap \U_1 = C_1$.
		
		\item $A \imp b$ is an implication of $\is[\U_1,\U_2]$.
		Then $A \subseteq \U_1$ and $b \in \U_2$ since we assumed that every 
		implication 
		of 
		$\is[\U_1, \U_2]$ has its premise in $\U_1$ and its conclusion in $\U_2$ . 
		Therefore, $b \notin X_{i + 1} \cap \U_1$. 
	\end{enumerate}
	Consequently $X_{i + 1} \setminus X_i \subseteq \U_2$, from which we deduce 
	that 
	$X_{i + 
		1} \cap \U_1 = C_1$, finishing the induction.
	Applying the result on $X_k = \cl(C_1)$, $\cl(C_1) \cap \U_1 = C_1$ follows.
	So $C_1 \in \cs \colon \U_1$ and $\cs_1 \subseteq  \cs \colon \U_1$. 
	The reverse inclusion holds by item \emph{1}.
	As for $\cs_2$, we have $\cs_2 \subseteq \cs$ as $A \subseteq \U_1$ for every 
	implication 
	$A \imp b$ of $\is[\U_1, \U_2]$.
\end{proof}

According to Theorem \ref{thm:trad:H-decomposition-cs} item \emph{1}, every 
closure 
system is a subset of the product of its H-factors closure systems. 
So it is possible to compute $\cs_1$ and $\cs_2$ in parallel for every split $(\U_1, \U_2)$ in the $\is$-tree, and then use the bipartite implicational base 
$\is[\U_1,\U_2]$ 
to compute $\cs$. 
But this strategy is expensive, since the size of $\cs_1$ and $\cs_2$ may be 
exponential in the size of $\cs$. 

\begin{example}
	Let $\U = \{u_1, \dots, u_k, x, y \}$ for some $k \in \cb{N}$ and let 
	$\is = \bigcup \{\{ \allowbreak u_i u_j \imp x, u_i u_j \imp y\} \mid 1 \leq 
	i, j, 
	\leq k, i \neq 
	j 
	\} \cup \{xy \imp u_i \mid 1 \leq i \leq k\}$.
	Clearly, the unique possible split is $(\U \setminus \{x, y\}, \{x, y\})$.
	Since $\is[\U \setminus \{x, y\}]$ is empty, its associated closure system is 
	Boolean 
	and 
	has $2^k$ elements.
	However, $\cs = \{v \mid v \in  \U\} \cup \{\{u, v\} \mid \{u, v\} \in (\U 
	\setminus 
	\{x, 
	y \}) \times \{x, y\}\} \cup \{\emptyset, \U\}$ so that $\card{\cs} = 3k + 4$.
\end{example}

However, this exponential reduction cannot occur when the sub-closure systems 
$\cs_1$ and $\cs_2$ appear as traces of $\cs$.

To conclude this section, we relate H-decomposition to the subdirect product 
decomposition \cite{ganter2012formal, gratzer2011lattice}.
In some cases, irreducible factors for splits are also subdirectly irreducible.
For instance, the closure systems depicted in Figure 
\ref{fig:trad:subdirect-2}(a), (b) and (c) are both subdirectly irreducible and 
irreducible H-factors.
\begin{figure}[h!]	
	\centering
	\includegraphics[scale=0.9, page=2]{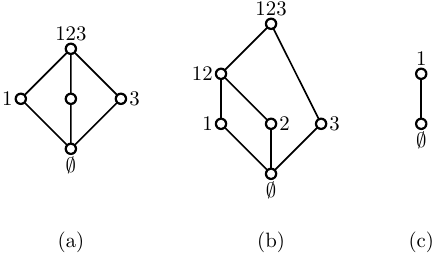}%
	\caption{Subdirectly irreducible H-indecomposable factors.}
	\label{fig:trad:subdirect-2}
\end{figure}
However, there are also closure systems that are subdirectly irreducible, and still admit 
a split.
Consider the closure system $\cs$ over $\U = \{1, 2, 3\}$ in Figure 
\ref{fig:trad:subdirect}
encoded by the implicational base $\{ 2 \imp 1, 13 \imp 2\}$. 
\begin{figure}[h!]	
	\centering
	\includegraphics[scale=0.9, page=1]{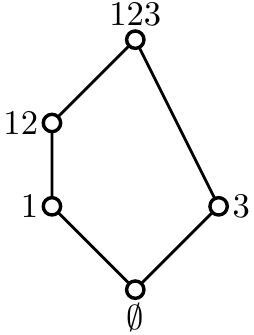}%
	\caption{A subdirectly irreducible closure system with a split.}
	\label{fig:trad:subdirect}
\end{figure}
It is known that it cannot be decomposed using the subdirect product.  
Clearly $\is$ is not premise-connected and $\U_1=\{1,3\}$ and $\U_2=\{2\}$ is 
the 
unique split where $\cs_1=\{\emptyset, 1, 3, 13\}$ and $\cs_2=\{\emptyset, 2\}$ are 
traces.
In this case though, $\cs$ is not a sublattice of $\cs_1 \times \cs_2$, since 
$\{1, 3\}$, the 
upper bound of $1$ and $3$ in $\cs_1 \times \cs_2$ is not preserved in $\cs$. 
Hence, we end the section with the following.

\begin{corollary} \label{cor:trad:subdirect}
	The closure system associated to an implicational base $\is$ is included in 
	the direct product of its H-factors. 
\end{corollary}

\begin{proof} 
	This follows from Theorem \ref{thm:trad:H-decomposition-cs}, item \emph{1} 
	and 
	the 
	fact that a closure system is closed under intersection. 
\end{proof}

In the next section, we pay more attention to particular splits called  
\emph{acyclic}.
We show how they can be applied to the problem of translating between the 
representations 
of a closure system.

\section{Closure systems with acyclic splits}
\label{sec:trad:acyclic}

In this section, we give a characterization of closure systems with acyclic splits.
Then, we derive a recursive expression of their meet-irreducible elements.
Finally, we devise an algorithm solving \csmc{CCM} in the case of acyclic 
splits.
To illustrate our results, we will use the following running example all along the 
section.

\begin{example}[Running example]
	Let $\U = \{1, 2, 3, 4, 5, 6\}$ and $\is = \{12 \imp 3, \allowbreak 13 \imp 4, 
	23 
	\imp 5, \allowbreak 2 \imp 4, \allowbreak 1 \imp 5, 5 \imp 6, 4 \imp 6\}$.
	We represent $\is$ and its associated closure system $\cs$ in Figure 
	\ref{fig:trad:over-ex-sig}.
	
	\begin{figure}[h!]
		\centering 
		\includegraphics[scale=0.9, page=1]{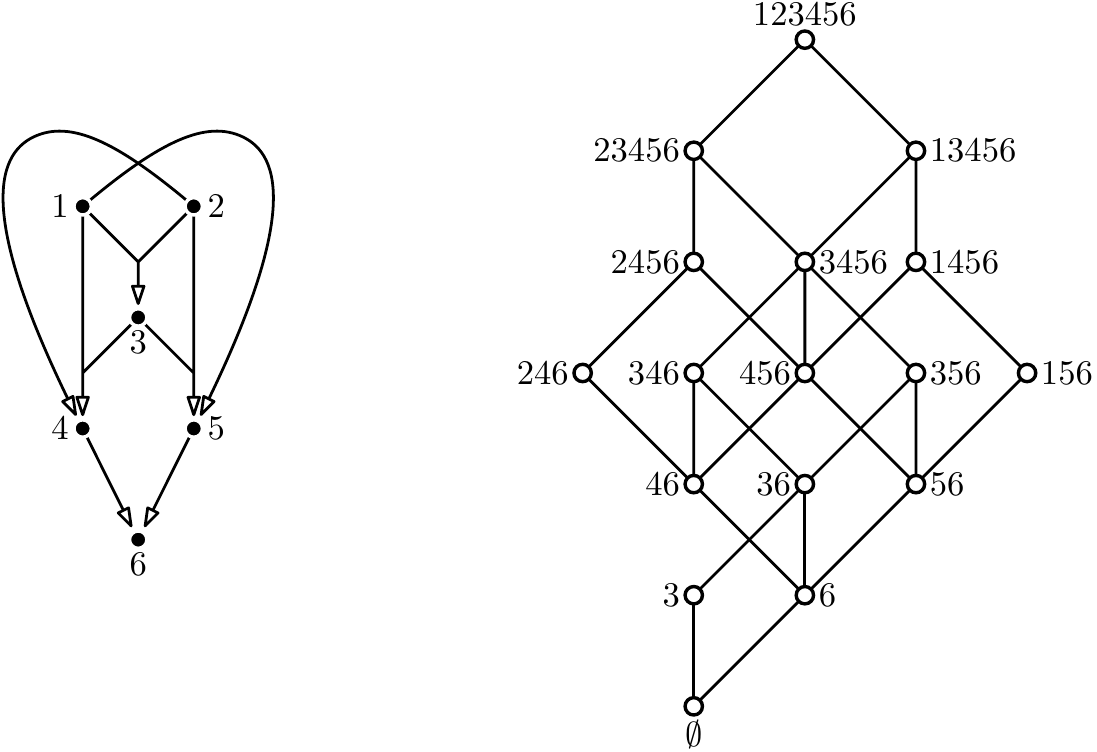}%
		\caption{An implicational base and its associated closure system.}
		\label{fig:trad:over-ex-sig}
	\end{figure}
	
	The bipartition $\U_1 = \{1, 2, 3\}$ and $\U_2 = \{4, 5, 6\}$ is an 
	\emph{acyclic} 
	split of $\is$ and $\cs$: every implications has its premise included in 
	$\U_1$ 
	and its conclusion in $\U_2$.
	We have $\is[\U_1] = \{12 \imp 3\}$, $\is[\U_2] = \{4 \imp 6, 5 \imp 6 \}$ 
	and $\is[\U_1, \U_2] = \{13 \imp 4, 2 \imp 4, 23 \imp 5, 1 \imp 5\}$.
	
\end{example}

We formally introduce \emph{acyclic split} of an implicational base $\is$.
They are a restriction of a split $(\U_1, \U_2)$ where all implications of 
$\is[\U_1, 
\U_2]$ have to go from $\U_1$ to $\U_2$, \ie they satisfy condition 
\emph{3} or \emph{4} of Theorem \ref{thm:trad:H-decomposition-cs}.
The definition of acyclic split for implicational bases extends to closure systems.

\begin{definition}[Acyclic split]
	Let $\is$ be an implicational base over $\U$ and $(\U_1, \U_2)$ a split of 
	$\is$.
	The split $(\U_1, \U_2)$ is \emph{acyclic} if for every $A \imp b \in 
	\is[\U_1, 
	\U_2]$, 
	$A \subseteq \U_1$.
\end{definition}

\begin{definition}[Acyclic split of a closure system]
	Let $\cs$ be a closure system over $\U$ and let $(\U_1, \U_2)$ be a 
	non-trivial 
	bipartition of $\U$ such that $\U_2 \in \cs$.
	Then, $(\U_1, \U_2)$ is an \emph{acyclic split} of $\cs$ if there exists an 
	implicational 
	base $\is$ for $\cs$ with acyclic split $(\U_1, \U_2)$.
\end{definition}

\subsection{Acyclic split of a closure system}
\label{subsec:trad:acyc-split-cs}

Let $\is$ be an implicational base over $\U$ with acyclic split $(\U_1, \U_2)$.
Let $\cs$ be its corresponding closure system.
We first show how to construct $\cs$ from $\cs_1$, the closure system associated 
to 
$\is[\U_1]$, $\cs_2$, the closure system of $\is[\U_2]$ and the implications in
$\is[\U_1, \U_2]$.

We draw intuition from the particular case where $\is[\U_1, \U_2] = \emptyset$.
According to Theorem \ref{thm:trad:H-decomposition-cs}, $\cs$ is the direct 
product 
of $\cs_1$ and $\cs_2$, that is $\cs = \{C_1 \cup C_2 \mid C_1 \in \cs_1, C_2 \in 
\cs_2\}$.
Intuitively, $\cs$ is obtained by \textit{``extending''} each closed set of 
$\cs_2$ with 
a copy of $\cs_1$ (see the left part of Figure \ref{fig:trad:over-construct}).
This point of view will be particularly well-suited for us, and naturally leads to 
the 
following definition.

\begin{definition}
	Let $\cs$ be a closure system over $\U$, $(\U_1, \U_2)$ be a non-trivial 
	bipartition 
	of 
	$\U$ such that $\U_2 \in \cs$.
	Let $C_2 \in\cs$, $C_2 \subseteq \U_2$ and $C \in \cs$.
	We say that $C$ is an \emph{extension} of $C_2$ with respect to $\U_2$ if $C 
	\cap 
	\U_2 = 
	C_2$.
	We denote by $\csf{Ext}(C_2)$ the extensions of $C_2$ in $\cs$.
	The trace $\csf{Ext}(C_2)$ on $\U_1$ is written $\csf{Ext}(C_2) \colon \U_1$.
\end{definition}

In our definition, $\U_2$ is closed.
Therefore, for every $C \in \cs$, $C \cap \U_2$ is also closed.
We deduce that $C$ belongs to the extension of a unique closed set $C_2$ included 
in 
$\U_2$.
As a consequence, we can write $\cs$ as the (disjoint) union of its extensions 
with 
respect to $\U_2$, \ie 
\[ \cs = \bigcup_{C_2 \in \cs, C_2 \subseteq \U_2} \Ext(C_2) \]
This definition of extensions allows to formally express the intuition that the 
direct 
product of $\cs_1$ and $\cs_2$ (when $\is[\U_1, \U_2] = \emptyset$) is obtained by 
extending each closed set of $\cs_2$ with a copy of $\cs_1$.
Indeed, we have $\cs = \bigcup_{C_2 \in \cs_2} \Ext(C_2)$ with the particularity 
that the trace of $\Ext(C_2)$ on $\U_1$ is exactly $\cs_1$ for every $C_2 \in 
\cs_2$.
This construction is illustrated on the left of Figure 
\ref{fig:trad:over-construct}.

\begin{figure}[h!]
	\centering 
	\includegraphics[scale=0.8, page=2]{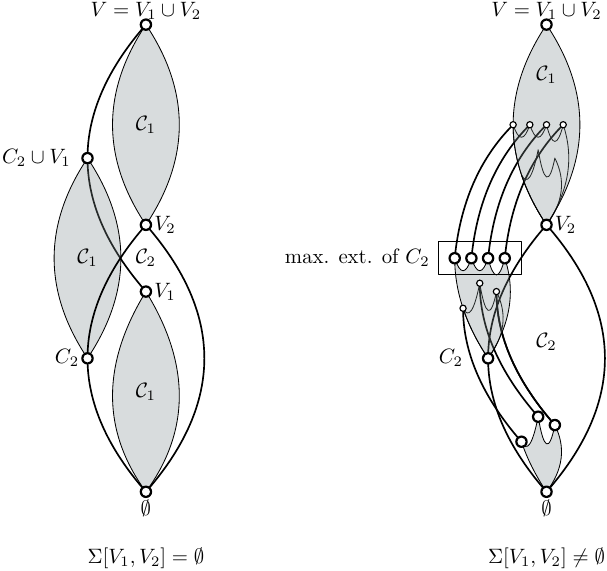}%
	\caption{Building of $\cs$ with an acyclic split: on the left, the case where 
		$\is[\U_1, \U_2] = \emptyset$ (direct product). On the right, the more 
		general 
		case where 
		$\is[\U_1, \U_2] \neq \emptyset$ (increasing extensions).}
	\label{fig:trad:over-construct}
\end{figure}

In the more general case where $\is[\U_1, \U_2]$ is nonempty, we show that 
the extensions of $\cs_2$ are no longer full copies of $\cs_1$, but increasing 
copies of 
ideals of $\cs_1$, as illustrated on the right side of Figure 
\ref{fig:trad:over-construct}.
We begin with the following proposition, which characterizes extensions with the 
bipartite 
set of implications $\is[\U_1, \U_2]$.

\begin{proposition} \label{prop:trad:charac-ext}
Let $\is$ be an implicational base over $\U$ with acyclic split $(\U_1, \U_2)$.
	Let $C_2 \in \cs_2$ and $C_1 \subseteq \U_1$.
	Then, $C = C_1 \cup C_2$ is an extension of $C_2$ if and only if $C_1 \in 
	\cs_1$ and for each implication $A \imp b$ in $\is[\U_1, \U_2]$, $A \subseteq C_1$ 
	implies $b \in C_2$. 
\end{proposition}

\begin{proof}
	We begin with the only if part.
	Let $C_1$ be a subset of $\U_1$ such that let $C_1$ be a closed set of $\cs_1$ 
	such 
	that 
	$C_1\cup C_2$ is an extension of $C_2$.
	By Theorem \ref{thm:trad:H-decomposition-cs}, $\cs \subseteq \cs_1 \times 
	\cs_2$ so 
	that 
	for every $C_1 \subseteq \U_1$ such that $C_1 \cup C_2 \in \cs$, $C_1 \in 
	\cs_1$ 
	holds.
	Now let $A \imp b \in \is[\U_1, \U_2]$.
	If $A \subseteq C_1$, it must be that $b \in C_2$ since we would contradict 
	$C_1\cup 
	C_2 
	\in \cs$ otherwise.
	
	We move to the if part.
	Let $C_1$ be a closed set of $\cs_1$ and $C_2$ a closed set of $\cs_2$ such 
	that for 
	each implication $A \imp b$ in $\is[\U_1, \U_2]$, $A \subseteq C_1$ implies $b 
	\in 
	C_2$.
	We have to show that $C_1 \cup C_2$ is closed.
	Let $A \imp b$ be an implication of $\is$ with $A \subseteq C_1 \cup C_2$.
	As $(\U_1, \U_2)$ is an acyclic split of $\U$, we have two cases: either $A 
	\imp b$ 
	is in 
	$\is[\U_1, \U_2]$ or it is not.
	In the second case, assume $A \imp b$ belongs to $\is[\U_1]$.
	As $A \subseteq C_1 \cup C_2$, we have $A \subseteq C_1$.
	Furthermore, $C_1$ is closed for $\is[\U_1]$.
	Hence, $b \in C_1 \subseteq C_1\cup C_2$.
	The same reasoning can be applied if $A \imp b$ is in $\is[\U_2]$.
	Now assume $A \imp b$ is in $\is[\U_1, \U_2]$.
	We have that $A \subseteq \U_1$ by definition of an acyclic split.
	In particular, we have $A \subseteq C_1$ which entails $b \in C_2$ by 
	assumption.
	In any case, $C_1\cup C_2$ already contains $b$ for every implication $A \imp 
	b$ in 
	$\is$ such that $A \subseteq C_1\cup C_2$.
	Hence, $C_1\cup C_2$ is closed.
\end{proof}

We readily deduce from Proposition \ref{prop:trad:charac-ext} that $\Ext(\U_2) 
\colon 
\U_1$ is equal to $\cs_1$.
Proposition \ref{prop:trad:charac-ext} is also a step towards the next proposition.
It settles the fact that in an acyclic split, extensions coincide with ideals of 
$\cs_1$.

\begin{proposition} \label{prop:trad:ideals}
Let $\is$ be an implicational base over $\U$ with acyclic split $(\U_1, \U_2)$.
	Let $C_1\in \cs_1$, $C_2 \in \cs_2$. 
	If $C_1\cup C_2$ is an extension of $C_2$, then for every $C_1' \in \cs_1$ 
	such that 
	$C_1' \subseteq C_1$, $C_1' \cup C_2$ is also an extension of $C_2$.
\end{proposition}

\begin{proof}
	Let $C_1 \in \cs_1, C_2 \in \cs_2$ such that $C_1 \cup C_2 \in \cs$.
	Let $C_1' \in \cs_1$ such that $C_1' \subseteq C_1$.
	As $C_1\cup C_2$ is an extension of $C_2$, for each $A \imp b$ in $\is[\U_1, 
	\U_2]$ 
	such 
	that $A \subseteq C_1$, we have $b \in C_2$ by Proposition 
	\ref{prop:trad:charac-ext}. 
	Since $C_1' \subseteq C_1$, this condition holds in particular if $A \subseteq 
	C_1'$.
	Applying Proposition \ref{prop:trad:charac-ext}, we deduce that $C_1' \cup 
	C_2$ is 
	closed.
\end{proof}

In fact, the preceding proposition can be further strengthened.
Not only extensions of $\cs_2$ correspond to ideals of $\cs_1$, but they are 
increasing.
That is, if $C_1$ contributes to an extension of $C_2$, it will also contribute to 
an 
extension of any closed set $C_2' \in \cs_2$ including $C_2$.

\begin{lemma} \label{lem:trad:hereditary}
Let $\is$ be an implicational base over $\U$ with acyclic split $(\U_1, \U_2)$.
	Let $C_2, C_2' \in \cs_2$ such that $C_2 \subseteq C_2'$.
	Then $\Ext(C_2) \colon \U_1 \subseteq \Ext(C_2') \colon \U_1$.
\end{lemma}

\begin{proof}
	We need to show that for every $C_2, C_2' \in \cs_2$ such that $C_2 \subseteq 
	C_2'$, 
	if 
	$C_1 \cup C_2 \in \cs$ for some $C_1 \subseteq \U_1$, we also have $C_1 \cup 
	C_2' \in 
	\cs$.
	Observe that due to Proposition \ref{prop:trad:charac-ext}, $C_1 \in \cs_1$.
	As $C_1 \cup C_2$ is an extension of $C_2$, for every implication $A \imp b$ 
	of 
	$\is[\U_1, \U_2]$ such that $A \subseteq C_1$, we have $b \in C_2 \subseteq 
	C_2'$ by 
	Proposition \ref{prop:trad:charac-ext}.
	Therefore, $C_1\cup C_2'$ is indeed an extension of $C_2'$.
\end{proof}

\begin{corollary} \label{cor:trad:prec-hered}
Let $\is$ be an implicational base over $\U$ with acyclic split $(\U_1, \U_2)$.
	Let $C_2, C_2' \in \cs_2$ such that $C_2 \prec C_2'$ and let $C_1 \in \cs_1$ 
	such 
	that 
	$C_1 \cup C_2 \in \cs$.
	Then $C_1 \cup C_2' \in \cs$ and $C_1 \cup C_2 \prec C_1 \cup C_2'$.
\end{corollary}

\begin{proof}
	The fact that $C_1 \cup C_2'$ is closed follows from Lemma 
	\ref{prop:trad:hereditary}.
	By Theorem \ref{thm:trad:H-decomposition-cs}, $\cs \subseteq \cs_1 \times 
	\cs_2$ so 
	that 
	any closed set $C$ such that $C_1 \cup C_2 \subset C \subseteq C_1 \cup C_2' $ 
	satisfies 
	$C \cap \U_2 \in \cs_2$.
	Since $C_2 \prec C_2'$ in $\cs_2$, $C = C_1 \cup C_2'$ follows.
\end{proof}

Thus, we have shown that if $(\U_1, \U_2)$ is an acyclic split of $\is$, $\cs$ can 
be 
constructed by extending each closed set $C_2$ of $\cs_2$, with an ideal of 
$\cs_1$, in 
an increasing fashion.
This construction is illustrated in Figure \ref{fig:trad:over-construct} and in 
Figure 
\ref{fig:trad:over-ex-ext} on an example.
In the next theorem, we demonstrate that this construction by increasing 
extensions is in 
fact a characterization of acyclic splits.

\begin{theorem} \label{thm:trad:charac-acyclic-split}
	Let $\cs$ be a closure system over $\U$ and $(\U_1, \U_2)$ be a non-trivial 
	bipartition 
	of $\U$ such that $\U_2 \in \cs$.
	Let $\cs_1 = \ftr \U_2 \colon \U_1$ and $\cs_2 = \idl \U_2$.
	Then, $(\U_1, \U_2)$ is an acyclic split for $\cs$ if and only if for every 
	$C_2, 
	C_2' 
	\in \cs_2$ such that $C_2 \subseteq C_2'$, we have $\Ext(C_2) \colon \U_1 
	\subseteq 
	\Ext(C_2') \colon \U_1$.
\end{theorem}

\begin{proof}
	The only if part follows from Lemma \ref{lem:trad:hereditary}.
	To show the if part, we build an implicational base $\is$ with the acyclic 
	split 
	$(\U_1, 
	\U_2)$.
	Beforehand, we outline the main ideas:
	\begin{itemize}
		\item $\is$ should contain an implicational base for $\cs_2$ as it is an 
		ideal of 
		$\cs$;
		\item $\is$ should also include an implicational base for $\cs_1$ since it 
		is a 
		filter of $\cs$ and $\is$ must respect the split $(\U_1, \U_2)$;
		\item $\is$ must describe, for each $C_2 \in \cs_2$, which closed sets of 
		$\cs_1$ contribute to extensions of $C_2$ or not.
		The most direct way to express this relationship is to explicitly write it 
		in 
		$\is$ 
		by putting implications $C_1 \imp \cl(C_1) \cap \U_2$, if $C_1$ does not 
		participate in an extension of $C_2$.
	\end{itemize}
	Actually, we can readily optimize the last item.
	Indeed, since the property of not contributing to an extension is monotone, it 
	is 
	sufficient to put an implication $C_1 \imp \cl(C_1) \cap \U_2$ if $C_1$ is a 
	minimal 
	closed set of $\cs_1$ which does not yield an extension of $C_2$.
	
	With these ideas in mind, we proceed now to the proof.
	Let $\cs_1 = \ftr \U_2 \colon \U_1$ and $\cs_2 = \idl \U_2$.
	Observe that both $\cs_1$ and $\cs_2$ are closure systems.
	We aim to construct an implicational base $\is$ representing $\cs$ with 
	acyclic split 
	$(\U_1, \U_2)$.
	
	First, we prove that $\cs \subseteq \cs_1 \times \cs_2$.
	Let $C \in \cs$ and let $C_1 = C \cap \U_1$ and $C_2 = C \cap \U_2$.
	As $C$ and $\U_2$ are closed in $\cs$ we deduce that $C_2 \in \cs_2$ and hence 
	that 
	$C 
	\in \Ext(C_2)$.
	As $C_2 \subseteq \U_2$, we have $\Ext(C_2) \colon \U_1 \subseteq \Ext(\U_2) 
	\colon 
	\U_1$ with $\Ext(\U_2) \colon \U_1 = \cs_1$ by assumption.
	Hence $C_1 \in \cs_1$.
	We deduce that $\cs \subseteq \cs_1 \times \cs_2$.
	
	Now, let $\is[\U_1]$ be an implicational base for $\cs_1$, $\is[\U_2]$ an 
	implicational 
	base for $\cs_2$ and let 
	\[ \is[\U_1, \U_2] = \{C_1 \imp \cl(C_1) \cap \U_2 \mid C_1 \in \min(\cs_1 
	\setminus 
	\Ext(C_2) \colon \U_1) \text{ for some } C_2 \in \cs_2\}\]
	Finally we put $\is = \is[\U_1, \U_2] \cup \is[\U_1] \cup \is[\U_2]$.
	Clearly $(\U_1, \U_2)$ is an acyclic split for $\is$.
	We prove that $\is$ is an implicational base for $\cs$.
	Let $\cs_{\is}$ be the closure system associated to $\is$.
	
	To show that $\cs_{\is} \subseteq \cs$, we prove that $C \notin \cs$ entails 
	$C 
	\notin 
	\cs_{\is}$, for every $C \subseteq \U$.
	Let $C \subseteq \U$ such that $C \notin \cs$ and put $C_1 = C \cap \U_1$ and 
	$C_2 = 
	C 
	\cap \U_2$.
	First, assume that $C \notin \cs_1 \times \cs_2$. 
	Since $\cs \subseteq \cs_1 \times \cs_2$, $C \notin \cs$ readily holds.
	Then, $C_1 \notin \cs_1$ or $C_2 \notin \cs_2$ so that $C$ fails $\is[\U_1]$ 
	or 
	$\is[\U_2]$ and $C \notin \cs_{\is}$ holds.
	Now assume that $C \in \cs_1 \times \cs_2$ but $C \notin \cs$.
	By construction of $\cs$, we have that $C \notin \Ext(C_2)$, or equivalently, 
	$C_1 
	\notin 
	\Ext(C_2) \colon \U_1$.
	Let $C_1' \in \cs_1$ with $C_1' \subseteq C_1$ and  $C_1' \in \min(\cs_1 
	\setminus 
	\Ext(C_2) \colon \U_1))$.
	We show that $C$ fails the implication $C_1' \imp \cl(C_1') \cap \U_2$ of 
	$\is[\U_1, 
	\U_2]$.
	We have $\cl(C_1') \in \cs$ so that $\cl(C_1') \cap \U_2 \in \cs_2$ and $C_1' 
	\in 
	\Ext(\cl(C_1') \cap \U_2) \colon \U_1$.
	By assumption, for every closed set $C_2'' \in \cs_2$ such that $\cl(C_1') 
	\cap \U_2 
	\subseteq C_2''$, $\Ext(\cl(C_1) \cap \U_2) \colon \U_1 \subseteq \Ext(C_2'') 
	\colon 
	\U_1$.
	Therefore, $C_1' \notin \Ext(C_2) \colon \U_1$ implies that $\cl(C_1') \cap 
	\U_2 
	\nsubseteq C_2$.
	Consequently, $C_1' \subseteq C_1 \subseteq C$ but $\cl(C_1') \cap \U_2 
	\nsubseteq C 
	\cap 
	\U_2 = C_2$.
	We deduce that $C \notin \cs_{\is}$, and hence that $\cs_{\is} \subseteq \cs$.
	
	Now we demonstrate that $\cs \subseteq \cs_{\is}$.
	Let $C \in \cs$ and put $C_1 = C \cap \U_1$, $C_2 = C \cap \U_2$.
	Recall that $\cs_2 = \idl \U_2$ and that $\is[\U_2]$ is an implicational base 
	for 
	$\cs_2$.
	Therefore, $C_2 \in \cs_2$ and $C$ is a model of $\is[\U_2]$ since $C_2 
	\subseteq C$.
	Now, because $C_2 \subseteq \U_2$, we have $\Ext(C_2) \colon \U_1 \subseteq 
	\Ext(\U_2) 
	\colon \U_1 = \cs_1$ by assumption.
	Moreover, $\is[\U_1]$ is an implicational base for $\cs_1$.
	Consequently, we obtain that $C_1 \in \cs_1$ and hence that $C$ is a model for 
	$\is[\U_1]$.
	It remains to show that $C$ also models $\is[\U_1, \U_2]$.
	But this is clear as $C = \cl(C)$ and each implication $C_1 \imp \cl(C_1') 
	\cap \U_2$ 
	of 
	$\is[\U_1, \U_2]$ satisfies $\cl(C_1') \cap \U_2 \subseteq \cl(C_1')$ .
	Hence, $C_1' \subseteq C$ implies that $\cl(C_1') \subseteq C$.
	Consequently, $\cs \subseteq \cs_{\is}$ and $\cs = \cs_{\is}$ holds, 
	concluding the 
	proof.
\end{proof}

\begin{example}[Running example]
	The closure system $\cs_1$ associated to $\is[\U_1] = \{12 \imp 3\}$ is given 
	on the 
	left 
	of Figure \ref{fig:trad:over-ex-subcs}.
	On the right, we give $\cs_2$, the closure system of $\is[\U_2] =  \{4 \imp 6, 
	5 \imp 
	6\}$.
	
	\begin{figure}[h!]
		\centering
		\includegraphics[scale=0.9, page=3]{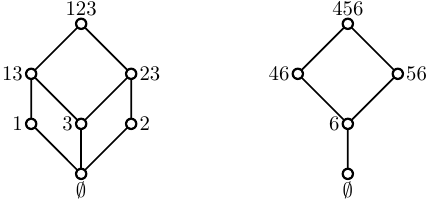}%
		\caption{The closure systems $\cs_1$ and $\cs_2$.}
		\label{fig:trad:over-ex-subcs}
	\end{figure}
	
	The construction of $\cs$ using extensions with respect to $\cs_1$ and $\cs_2$ 
	suggested 
	by Theorem \ref{thm:trad:charac-acyclic-split} is highlighted in Figure 
	\ref{fig:trad:over-ex-ext}.
	For instance, the extensions of $6$ are $\emptyset$ and $36$.
	Remark that $\emptyset$ and $3$ also contribute to the extensions $46$, $346$ 
	of $46$.
	Moreover, $346$ is a maximal extension of $46$, along with $246$.
	Finally, the extensions of $456$ (that is, $\U_2$) coincide with $\cs_1$.
	
	\begin{figure}[h!]
		\centering
		\includegraphics[scale=0.9, page=4]{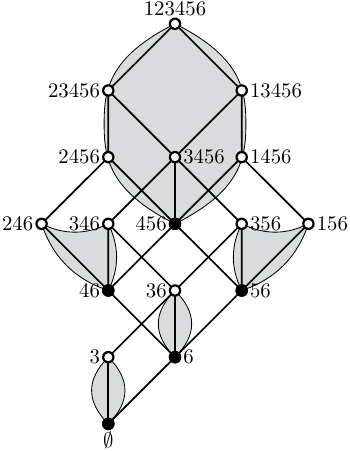}%
		\caption{The closure $\cs$ constructed from $\cs_1$ and $\cs_2$ (black 
		dots are 
			closed set of $\cs_2$).}
		\label{fig:trad:over-ex-ext}
	\end{figure}
	
\end{example}

In the particular case where $\cs$ is a direct product of $\cs_1, \cs_2$, the pair 
$(\U_1, \U_2)$ becomes a strong decomposition pair of \cite{libkin1993direct}.
It is worth noticing that Theorem \ref{thm:trad:charac-acyclic-split} hints a 
strategy to 
recursively compute the meet-irreducible elements of $\cs$.
This is the aim of the next subsection.

\subsection{The meet-irreducible elements of a closure system with acyclic split}
\label{subsec:trad:acyc-split-meet}

Now we use Theorem \ref{thm:trad:charac-acyclic-split} to obtain a recursive
expression of $\M$, the meet-irreducible elements of $\cs$ in terms of $\M_1$ and 
$\M_2$, 
the meet-irreducible elements of $\cs_1$ and $\cs_2$ respectively.
We prove that the decomposition of $\cs$ with extensions captures the 
structure of $\M$.
Again, we start from the case of the direct product.
This result has already been formulated in lattice theory, for instance in 
\cite{davey2002introduction}. 
For convenience, we rewrite it in our terms.

\begin{proposition}[\cite{davey2002introduction}, p. 119] \label{prop:trad:meet-direct}
	Let $\cs_1$ and $\cs_2$ be two closure systems over $\U_1$ and $\U_2$ (resp.) 
	where 
	$\U_1$ and $\U_2$ are disjoint.
	Let $\cs = \cs_1 \times \cs_2$.
	Then $\M = \{M_1 \cup \U_2 \mid M_1 \in \M_1\} \cup \{M_2 \cup \U_1 \mid M_2 
	\in \M_2\}$. 
\end{proposition}

\begin{figure}[h!]
	\centering 
	\includegraphics[scale=0.8, page=3]{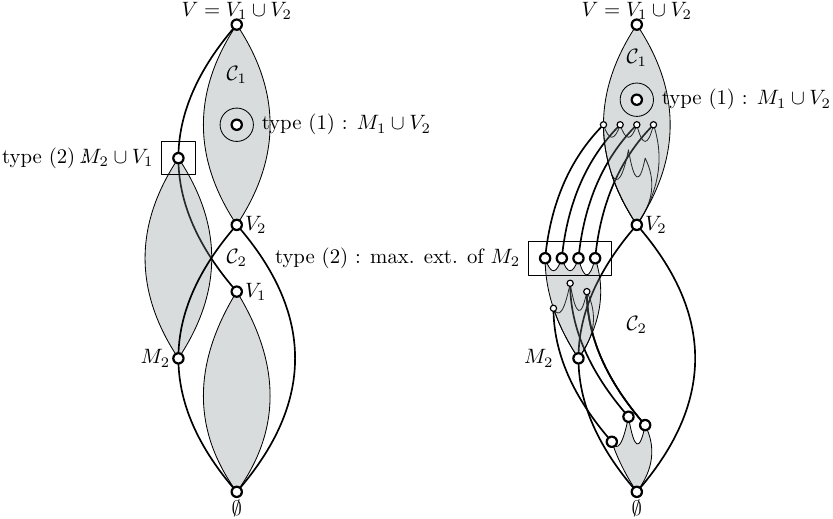}%
	\caption{Meet-irreducible elements of $\cs$ with an acyclic split: on the 
	left, the 
		direct product. On the right, the case of acyclic splits in general.}
	\label{fig:trad:over-meet}
\end{figure}

If we adopt the point of view of extensions with respect to $\cs_2$, as in the 
previous 
subsection, the meet-irreducible elements of $\cs_1 \times \cs_2$ can be 
partitioned into 
two classes:
\begin{itemize}
	\item[(1)] those belonging to extensions of $\U_2$, that is $\{M_1 \cup \U_2 
	\mid M_1 
	\in 
	\M_1\}$;
	\item[(2)] meet-irreducible elements of $\M_2$ which we extended with $\U_1$, 
	that is 
	$\{ M_2\cup \U_1 \mid M_2 \in \M_2 \}$.
	Observe that $M_2 \cup \U_1$ is the unique inclusion-wise maximal extension of 
	$M_2$, 
	for each 
	$M_2 \in \M_2$.
\end{itemize}
This construction is illustrated on the left part of Figure 
\ref{fig:trad:over-meet}.

We show next that when $\cs$ has an acyclic split $(\U_1, \U_2)$ but it is not the 
direct 
product of $\cs_1$ and $\cs_2$, the structure of $\M$ preserves this partitioning: 
\begin{itemize}
	\item[(1)] $\{M_1 \cup \U_2 \mid M_1 \in \M_1\}$ 
	remains unchanged;
	\item[(2)] $\{ M_2\cup \U_1 \mid M_2 \in \M_2 \}$ is adapted to replace 
	$M_2 \cup \U_1$ by the possible maximal extensions of elements of $\M_2$.
\end{itemize}
This construction is represented on the right of Figure \ref{fig:trad:over-meet}.
Let $\cs$ be a closure system with acyclic split $(\U_1, \U_2)$.
Again, let $\cs_1 = \ftr \U_2 \colon \U_1$ and $\cs_2 = \idl \U_2$.
We begin with the following two lemmas.

\begin{lemma} \label{lem:trad:meet-1}
	Let $\cs$ be a closure system over $\U$ with acyclic split $(\U_1, \U_2)$.
	Let $C_2 \in \cs_2, C_2 \neq \U_2$ and $C_1\in \cs_1$ such that $C_1\cup C_2$ 
	is a 
	non-maximal extension of $C_2$. Then $C_1\cup C_2 \notin \M$.
\end{lemma}

\begin{proof}
	Let $C_2 \in \cs_2, C_2 \neq \U_2$ and $C_1\in \cs_1$ such that $C_1\cup C_2$ 
	is a 
	non-maximal extension of $C_2$.
	As $C_2 \neq \U_2$, there exists at least one closed set $C_2' \in \cs_2$ such 
	that 
	$C_2 \prec C_2'$.
	By Corollary \ref{cor:trad:prec-hered} we have that $C_1\cup C_2 \prec C_1\cup 
	C_2'$ 
	in $\cs$.
	Furthermore, $C_1\cup C_2$ is not a maximal extension of $C_2$. 
	Therefore, there exists a closed set $C_1'$ in $\cs_1$ such that $C_1 \prec 
	C_1'$ 
	and $C_1' \cup C_2 \in \cs$.
	As $\cs \subseteq \cs_1 \times \cs_2$ by Theorem 
	\ref{thm:trad:charac-acyclic-split} 
	and extensions are increasing by Lemma \ref{lem:trad:hereditary}, it follows 
	that 
	$C_1 
	\cup 
	C_2 \prec C_1' \cup C_2$ in $\cs$ with $C_1\cup C_2' \neq C_1' \cup C_2$.
	Therefore, $C_1\cup C_2$ is not a meet-irreducible element of $\cs$.
\end{proof}

\begin{lemma} \label{lem:trad:meet-2}
	Let $\cs$ be a closure system over $\U$ with acyclic split $(\U_1, \U_2)$.
	Let $C_2 \in \cs_2$ such that $C_2 \neq \U_2$ and $C_2 \notin \M_2$. 
	Then $C \notin \M$ for every $C \in \Ext(C_2)$.
\end{lemma}

\begin{proof}
	Let $C_2 \in \cs_2$ such that $C_2 \neq \U_2$ and $C_2 \notin \M_2$.
	Let $C \in \Ext(C_2)$ and $C_1 = C \cap \U_1$.
	As $C_2 \notin \M_2$, it has at least two covers $C_2', C_2''$ in $\cs_2$.
	By Corollary \ref{cor:trad:prec-hered}, it follows that both $C_2' \cup C_1$ 
	and 
	$C_2'' 
	\cup C_1$ are covers of $C$ in $\cs$.
	Hence $C \notin \M$.
\end{proof}

These lemmas suggest that meet-irreducible elements of $\cs$ arise from maximal 
extensions of meet-irreducible elements of $\cs_2$. 
They might also come from meet-irreducible extensions of $\U_2$ since $\Ext(\U_2) 
\colon 
\U_1 = \cs_1$.
These ideas are proved in the following theorem, which characterize the 
meet-irreducible 
elements $\M$ of $\cs$ according to the two types we described.

\begin{theorem} \label{thm:trad:split-meet}
	Let $\cs$ be a closure system over $\U$ with acyclic split $(\U_1, \U_2)$. 
	Let $\cs_1 = \ftr \U_2 \colon \U_1$ and $\cs_2 = \idl \U_2$.
	The meet-irreducible elements $\M$ of $\cs$ satisfy $\card{\M} \geq \card{\M_1} + 
	\card{\M_2}$ and are subject to the following equality:
	\[ \M = \{M_1 \cup \U_2  \mid M_1 \in \M_1 \} \cup \{C \in \max(\Ext(M_2)) 
	\mid M_2 
	\in 
	\M_2\} \]
\end{theorem}

\begin{proof}
	First, $\{M_1 \cup \U_2 \mid M_1 \in \M_1 \} \subseteq \M$ follows from the 
	fact that 
	$\cs_1 = \ftr \U_2 \colon \U_1$.
	We prove that $\max(\Ext(M_2)) \subseteq \M$ for every $M_2 \in \M_2$.
	Let $M_2 \in \M_2$ and let $C$ be a maximal extension of $M_2$ with $C = C_1 
	\cup 
	M_2$.
	Since $M_2 \in \cs_2$, it has a unique cover $M_2'$ in $\cs_2$.
	By Corollary \ref{cor:trad:prec-hered}, we get $C \prec M_2' \cup C_1$ in 
	$\cs$.
	Let $C' \in \cs$ such that $C \subset C'$.
	Recall that $\cs \subseteq \cs_1 \times \cs_2$ follows from Theorem 
	\ref{thm:trad:charac-acyclic-split}, so that $C' \cap \U_1 \in \cs_1$ and $C' 
	\cap 
	\U_2 
	\in \cs_2$.
	Furthermore, $C \in \max(\Ext(M_2))$, therefore $C \subset C'$ implies that 
	$M_2 
	\subset 
	C' \cap \U_2$ and hence that $M_2' \subseteq C' \cap \U_2$ as $M_2 \in \cs_2$.
	Since $C_1\subseteq C' \cap \U_1$, we get $C \prec M_2' \cup C_1\subseteq C'$ 
	and 
	$C \in \M$ as it has a unique cover.
	
	Now we prove the other side of the equation.
	Let $M \in \M$.
	As $\cs \subseteq \cs_1 \times \cs_2$ since $(\U_1, \U_2)$ is an acyclic split 
	of 
	$\cs$, 
	$M \cap \U_2 \in \cs_2$ and we can distinguish two cases.
	Either $M \cap \U_2 = \U_2$ or $M \cap \U_2 \subset \U_2$.
	If $M \cap \U_2 = \U_2$ then $M$ is a meet-irreducible element of the closure 
	system 
	$\ftr \U_2$.
	Since $\ftr \U_2 \colon \U_1 = \cs_1$, we obtain that $M \cap \U_1 = M_1 \in 
	\M_1$.
	Now assume that $M \cap \U_2 \subset \U_2$.
	Let $M_1 = M \cap \U_1$ and $M_2 = M \cap \U_2$.
	Then by contrapositive of Lemma \ref{lem:trad:meet-1} we have that $M \in 
	\max(\Ext(M_2))$ as $M_2 \neq \U_2$.
	Similarly, we get $M_2 \in \M_2$ by Lemma \ref{lem:trad:meet-2}.
	The inequality $\card{\M} \geq \card{\M_1} + \card{\M_2}$ follows from the 
	description of 
	$\M$.
\end{proof}

\begin{example}[Running example]
	The meet-irreducible elements $\M_1$ of $\cs_1$ are $1$, $13$, $2$ and $23$.
	The meet-irreducible elements of $\cs_2$ are $\emptyset$, $46$ and 
	$56$.
	In Figure \ref{fig:trad:over-ex-meet} we highlight the two types of 
	meet-irreducible 
	elements of $\cs$, based on Theorem \ref{thm:trad:split-meet}.
	For instance $23456$ is of type (1) as it is obtained from the 
	meet-irreducible 
	element $23$ of $\cs_1$ and $\U_2$.
	Dually, $356$ is of type (2) because it is a maximal extension of the 
	meet-irreducible 
	element $56$ of $\cs_2$.
	
	\begin{figure}[h!]
		\centering
		\includegraphics[scale=0.8, page=5]{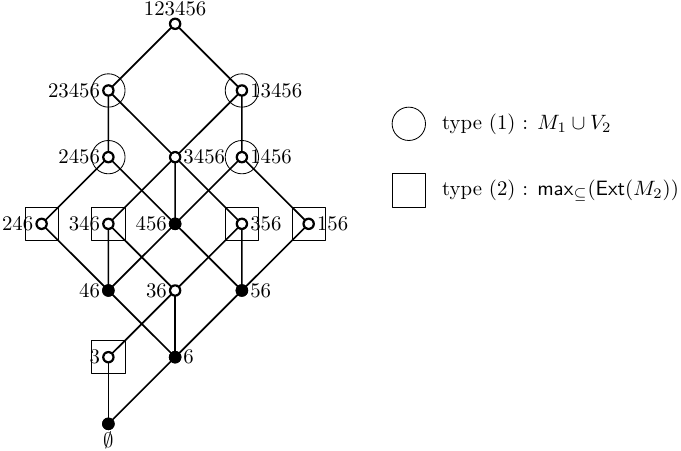}
		\caption{The two types of meet-irreducible elements in $\cs$ (black dots 
		are 
			closed 
			sets of $\cs_2$).}
		\label{fig:trad:over-ex-meet}
	\end{figure}
	
\end{example}

To conclude this section, we briefly discuss another characterization of acyclic 
splits based on Theorem \ref{thm:trad:charac-acyclic-split} and Theorem 
\ref{thm:trad:split-meet}.
Because extensions are hereditary, the extensions of $\M_2$ completely capture 
extensions of $\cs_2$.
In other words, if $C_2 \in \cs_2$ and $C_1$ contributes to an extension of $C_2$, 
then 
$C_1 \cup M_2$ is also an extension of $M_2$, for every $M_2 \in \M_2(C_2)$.
Therefore, $C_1 \cup C_2$ results from the intersection of the closed sets $M_2 
\cup C_1$, $M_2 \in \M_2(C_2)$.
We illustrate this idea in Figure \ref{fig:trad:cor-meet}. 


\begin{figure}[h!]
	\centering 
	\includegraphics[scale=0.9, page=1]{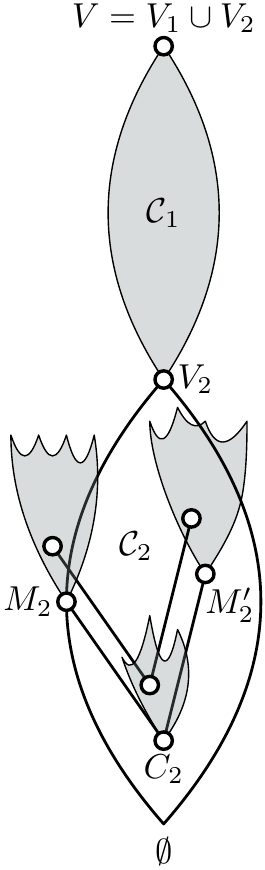}%
	\caption{Computing extensions of a closed set using extensions of 
	meet-irreducible 
		elements of $\cs_2$.}
	\label{fig:trad:cor-meet}
\end{figure}

\begin{corollary} \label{cor:trad:charac-acyclic-from-meet}
	Let $\cs$ be a closure system over $\U$ and $(\U_1, \U_2)$ a non-trivial 
	bipartition 
	of 
	$\U$ with $\U_2 \in \cs$.
	Let $\cs_1 = \ftr \U_2 \colon \U_1$ and $\cs_2 = \idl \U_2$.
	The pair $(\U_1, \U_2)$ is an acyclic split for $\cs$ if and only if for every 
	$C_2 
	\in 
	\cs_2$ and $C_2' \in \M_2(C_2) \cup \{\U_2\}$, $\Ext(C_2) \colon \U_1 
	\subseteq 
	\Ext(C_2') \colon \U_1$.
\end{corollary}

\begin{proof}
	The only if part follows from Theorem \ref{thm:trad:charac-acyclic-split}.
	Let $C_2, C_2' \in \cs_2$ with $C_2 \subseteq C_2'$.
	If $C_2 = \U_2$ or $C_2' = \U_2$, the fact that $\Ext(C_2) \colon \U_1 
	\subseteq 
	\Ext(C_2') \colon \U_1$ is clear.
	Assume that $C_2 \subseteq C_2' \subset \U_2$ so that $\M_2(C_2)$ and 
	$\M_2(C_2')$ 
	are 
	not empty.
	From $C_2 \subseteq C_2'$, we deduce $\M_2(C_2') \subseteq \M_2(C_2)$.
	Let $C \in \Ext(C_2)$ with $C_1 = C \cap \U_1$.
	Remark that $C_1 \in \cs_1$ holds by assumption.
	Moreover, for every $M_2 \in \M_2(C_2)$, we have $C_1 \cup M_2 \in \Ext(M_2)$.
	This holds in particular for every $M_2 \in \M_2(C_2')$ so that $\bigcap_{M_2 
	\in 
		\M_2(C_2')} (M_2 \cup C_1) = (\bigcap_{M_2 \in \M_2(C_2')} M_2) \cup C_1 = 
		C_2' 
	\cup C_1 
	\in \cs$.
	Consequently, $C_1 \cup C_2' \in \Ext(C_2')$ holds, concluding the proof.
\end{proof}

\subsection{Acyclic splits and \csmc{CCM}}
\label{subsec:trad:trad}

We apply Theorem \ref{thm:trad:split-meet} to the problem \csmc{CCM}.
Let $\cs$ be a closure system over $\U$ and $\is$ be an implicational base for 
$\cs$.
We assume that $\is$ has an acyclic split $(\U_1, \U_2)$.
According to Theorem \ref{thm:trad:split-meet}, computing $\M$ from $\M_1$ and 
$\M_2$ 
requires finding maximal extensions of every meet-irreducible element $M_2 \in 
\M_2$.

\Problem{Find Maximal Extensions (MaxExt)}
{A triple $\is[\U_1]$, $\is[\U_2]$, $\is[\U_1, \U_2]$ given by an acyclic split of 
an 
	implicational base $\is$, meet-irreducible elements $\M_1, \M_2$, and a closed 
	set 
	$C_2$ 
	of $\is[\U_2]$.}
{The maximal extensions of $C_2$ in $\cs$, \ie $\max(\Ext(C_2))$.}

This problem relates to the \emph{dualization in closure systems}.
Let $\cs$ be a closure system over $\U$ and $\Bm, \Bp$ two antichains of $\cs$.
We say that $\Bm$ and $\Bp$ are \emph{dual} in $\cs$ if $\idl \Bp \cup \ftr \Bm = 
\cs$ 
and $\idl \Bp \cap \ftr \Bm = \emptyset$.
The antichain $\Bp$ is referred to as the \emph{positive border}, while $\Bm$ is the 
\emph{negative border}.
Observe that $\Bp = \max(\{C \in \cs \mid C \notin \ftr \Bm\})$ and similarly $\Bm 
= 
\min(\{C \in \cs \mid C \notin \idl \Bp\})$ so that $\Bm$ is the unique negative 
border 
associated to $\Bp$, and vice-versa for $\Bp$.

We connect maximal extensions of a closed set with dualization.
Consider a closure system $\cs$ with acyclic split $(\U_1, \U_2)$.
Let $C_2 \in \cs_2$. 
Since $\Ext(C_2) \colon \U_1$ is an ideal of $\cs_1$, the antichain 
$\max(\Ext(C_2) 
\colon \U_1)$, we call it $\Bp$, has a dual antichain $\Bm$ 
in $\cs_1$.
We have $\Bm = \min(\cs_1 \setminus \Ext(C_2) \colon \U_1)$.
In words, $\Bm$ is the family of minimal closed sets of $C_1$ that are not 
participating 
in extensions of $C_2$.

\begin{proposition} \label{prop:trad:dual-ext}
	Let $\cs$ be a closure system over $\U$ with acyclic split $(\U_1, \U_2)$.
	Let $C_2 \in \cs_2$, and $C_1\in \cs_1$. Then, $C_1\in \Bm$ if and only if 
	$C_1 \in \min\{\cl_1(A) \mid A \imp b \in \is[\U_1, \U_2], b \notin C_2\}$.
\end{proposition}

\begin{proof}
	We show the if part.
	We denote by $\cl_1$ the closure operator associated to $\is[\U_1]$.
	Let $C_1\in \min\{\cl_1(A) \mid A \imp b \in \is[\U_1, \U_2], b\notin C_2\}$.
	We show that for any closed set $C_1' \subseteq C_1$ in $\cs_1$, $C_1'$ 
	contributes 
	to an extension of $C_2$.
	It is sufficient to show this property to the case where $C_1' \prec C_1$ as 
	$\Ext(C_2) \colon \U_1$ is an ideal of $\cs_1$ by Proposition 
	\ref{prop:trad:ideals}.
	Hence, consider a closed set $C_1'$ in $\cs_1$ such that $C_1' \prec C_1$.
	Note that such $C_1'$ exists since $\emptyset \in \cs_1$  and no implication 
	$A \imp 
	b$ 
	in $\is$ has $A = \emptyset$ so that $\emptyset \subset \cl_1(A)$ for any 
	implication 
	$A 
	\imp b$ of $\is[\U_1, \U_2]$ such that $b \notin C_2$.
	Then, by construction of $C_1'$, for any $A \imp b$ in $\is[\U_1, \U_2]$ such 
	that $b 
	\notin C_2$, we have $\cl_1(A) \nsubseteq C_1'$.
	As $\cl_1$ is a closure operator, it is monotone and
	$\cl_1(A) \nsubseteq \cl_1(C_1') = C_1'$ entails $A \nsubseteq C_1'$ for 
	any such implication $A \imp b$.
	Therefore $C_1' \in \Ext(C_2) \colon \U_1$ and $C_1\in \Bm$.
	
	We prove the only if part using contrapositive.
	Assume $C_1\notin \min\{\cl_1(A) \allowbreak \mid A \imp b\in 
	\is[\U_1, \U_2]$, $v \notin C_2\}$.
	We have two cases.
	First, for any implication $A \imp b$ in $\is[\U_1, \U_2]$ such that $b \notin 
	C_2$, 
	$\cl_1(A) \nsubseteq C_1$.
	Since $\cl_1$ is monotone and $C_1$ is closed in $\cs_1$, we have $A 
	\nsubseteq C_1$ 
	and 
	$C_1 \in \Ext(C_2) \colon \U_1$ by Lemma \ref{lem:trad:hereditary}.
	Hence $C_1\notin \Bm(C_2)$.
	In the second case, there is an implication $A \imp b$ with $b \notin C_2$ in 
	$\is[\U_1, 
	\U_2]$ such that $\cl_1(A) \subseteq C_1$ which implies $C_1\notin 
	\Ext(C_2) \colon \U_1$.
	If $\cl_1(A) \subset C_1$, then clearly $C_1\notin \Bm$ as $\cl_1(A)  
	\in \cs_1$ and $\cl_1(A) \notin \Ext(C_2) \colon \U_1$.
	Hence, assume that $C = \cl_1(A)$.
	Since $C_1 \notin \min\{\cl_1(A) \allowbreak \mid A \imp b \in 
	\is[\U_1, \U_2], b \notin C_2\}$ by hypothesis, there exists another 
	implication $A' 
	\imp 
	b'
	\in \is[\U_1, \U_2]$ such that $b' \notin C_2$ and $\cl_1(A') \subset C_1$.
	Hence $\cl_1(A') \notin \Ext(C_2) \colon \U_1$ and $C_1\notin \Bm$ as 
	it is not an inclusion-wise minimum closed set which does not belong to 
	$\Ext(C_2) 
	\colon \U_1$.
\end{proof}

We can build $\Bm$ in polynomial time from $\is$ using Proposition 
\ref{prop:trad:charac-ext} and $\is[\U_1, \U_2]$: we compute $\cl_1(A)$ for every 
implication $A \imp b$ in $\is[\U_1, \U_2]$ and we keep the closed sets (in 
$\cs_1$) that 
are inclusion-wise minimal.
Therefore, the problem \csmc{MaxExt} relates to the following generation version 
of 
dualization.

\Problem{Lower dualization in closure systems (LDual($\alpha$))}
{A representation $\alpha$ for a closure system $\cs$ over $\U$, an antichain 
$\Bm$ of 
	$\cs$}
{The antichain $\Bp$ dual to $\Bm$.}

When $\alpha$ is an implicational base $\is$ or the set of meet-irreducible 
elements 
$\M$, the problem \csmc{LDual}($\alpha$) is impossible to solve in 
output-polynomial 
time unless $\P = \NP$ \cite{babin2017dualization,defrain2020dualization}.
However, in \csmc{MaxExt} we have access to both $\is_1$ and $\M_1$ so that the 
version 
of \csmc{LDual} we have to consider is the one where $\alpha$ is both an 
implicational 
base and a set of meet-irreducible elements, that is \csmc{LDual}($\is, \M$).
This version of \csmc{LDual} is open, even if not harder than \csmc{SID} 
\cite{babin2017dualization}.
When $\is = \emptyset$, \ie when the closure system is Boolean, the problem 
reduces to 
hypergraph dualization.

Now, we describe an algorithm for solving \csmc{CCM} in the presence of acyclic 
splits.
First, we have $\card{\M} \geq \card{\M_1} + \card{\M_2}$ due to Theorem 
\ref{thm:trad:split-meet}.
Furthermore, each $M \in \M$ arise from a unique element of 
$M' \in \M_1 \cup \M_2$, and each $M' \in \M_1 \cup \M_2$ is used to construct at 
least 
one new meet-irreducible element $M \in \M$.
Therefore, the algorithm will output every meet-irreducible element only once.
Furthermore, the space needed to store intermediate solutions is bounded by the 
size of 
the output $\M$ which prevents an exponential blow up during the execution.

The algorithm proceeds as follows.
If $\is$ has no acyclic split, we use routines such as in
\cite{mannila1992design, beaudou2017algorithms} to compute $\M$.
When $\U$ is a singleton, the unique meet-irreducible to find is 
$\emptyset$ and hence no call to other algorithm is required.
Otherwise, we find an acyclic split $(\U_1, \U_2)$ of $\is$ and we recursively 
call the 
algorithm on $\is[\U_1]$ and $\is[\U_2]$.
Then, we compute $\M$ using $\is$, $\M_1$, $\M_2$ and by solving \csmc{MaxExt}.
Observe that it takes polynomial time in the size of $\is$ and $\U$ to compute an 
acyclic split, if it exists:
\begin{itemize}
	\item compute the premise-connected components of $\is$;
	\item construct a directed graph on these components, with an arc from a 
	component 
	$C_1$ to $C_2$ if there is an implication $A \imp b$ in $\is$ such that $A 
	\subseteq 
	C_1$ and $b \in C_2$;
	\item then, an acyclic split exists if and only if there are at least two 
	strongly 
	connected components, and each non-trivial bipartition of the strongly 
	connected 
	components will represent an acyclic split.
\end{itemize}
Thus, the algorithm \ctt{BuildTree} can be adapted to find a decomposition with 
acyclic splits or return \csf{FAIL} if not possible in polynomial time.

\begin{example}[Running example]
	First, we compute a decomposition of $\is$ in terms of acyclic splits.
	We obtain the $\is$-tree illustrated in Figure \ref{fig:trad:over-ex-tree}. 
	
	\begin{figure}[h!]
		\centering 
		\includegraphics[scale=1.0, page=2]{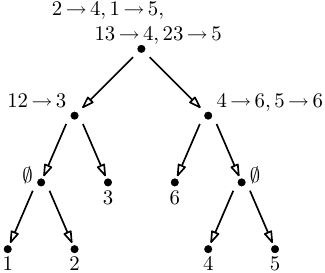}%
		\caption{The $\is$-tree of $\is$.}
		\label{fig:trad:over-ex-tree}	
	\end{figure}
	
	Then, we apply Theorem \ref{thm:trad:split-meet} bottom-up to construct the 
	the set $\M$ of meet-irreducible elements of $\cs$.
	This part is shown in  Figure \ref{fig:trad:over-ex-decomposition}.
	For readability, we highlighted at each step which closed sets are part of 
	$\cs_2$ 
	and 
	also the two types of meet-irreducible elements of Theorem 
	\ref{thm:trad:split-meet}.
	
	\begin{figure}[h!]
		\centering 
		\includegraphics[scale=0.7, page=6]{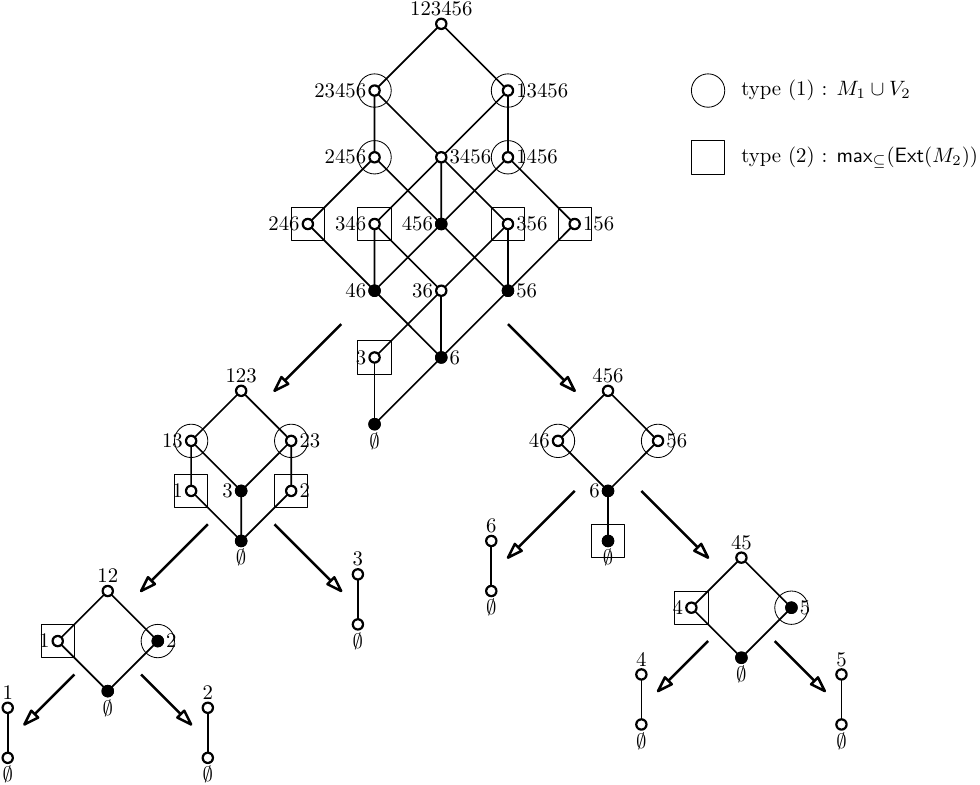}%
		\caption{Recursive computation of $\M$ using a decomposition by acyclic 
		splits.}
		\label{fig:trad:over-ex-decomposition}
	\end{figure}
	
\end{example}

To conclude, we derive a class of implicational bases where our strategy can be 
applied 
to obtain the meet-irreducible elements in output quasi-polynomial time.

\begin{theorem}
	Let $\is$ be an implicational base over $\U$.
	Assume there exists a full partition $\U_1, 
	\dots, \U_k$ of $\U$ such that for every implication $A \imp b \in \is$, $A 
	\subseteq 
	\U_i$ and $b \in \U_j$ for some $1 \leq i < j \leq k$.
	Then \csmc{CCM} can be solved in output-quasipolynomial time.
\end{theorem}

\begin{proof}
	Observe that $\is$ is acyclic in this case.
	Then, $\is$ can be hierarchically decomposed by $k - 1$ acyclic splits such 
	that the 
	implicational base on the left of the $i$-th split is $\is[\U_i] = \emptyset$ 
	and the 
	right-one $\is[\bigcup_{j > i} \U_j]$.
	Then, \csmc{MaxExt} reduces to hypergraph dualization, and
	we can compute $\M$ from $\is$ in output-quasipolynomial time using the 
	algorithm of 
	Fredman and Khachiyan \cite{fredman1996complexity}.
\end{proof}

The class of closure systems associated to these implicational bases generalizes 
both 
distributive closure systems and ranked convex geometries 
\cite{defrain2021translating} 
since an implicational base is ranked when it further satisfies the condition that 
$A 
\subseteq \U_i$ implies $b \in \U_{i + 1}$.

\section{Discussions and open problems}
\label{sec:trad:conc}

We conclude the paper with some discussions and open questions for future work.
Splits and more notably acyclic splits are decomposition methods based on the 
syntax of 
implications.
However, two equivalent implicational bases may not share the same (acyclic) 
splits.
In fact, it is even possible to find two equivalent implicational bases where one 
has an 
acyclic split, and not the other. 
This is demonstrated by the following example.

\begin{example}
	Let $\U = \{1, 2, 3, 4\}$ and $\is = \{1 \imp 4, 124 \imp 3, 3 \imp 4\}$.
	The unique possible split is $(124, 3)$ which is not acyclic.
	Observe that $\is$ is the Duquenne-Guigues base of the closure system.
	However, the implicational base $\is' = \{1 \imp 4, 12 \imp 3, 3 \imp 4\}$, 
	which is 
	clearly equivalent to $\is$ has an acyclic split being $(12, 34)$.
\end{example}

Note that the Duquenne-Guigues base is not of interest for finding acyclic splits 
as it 
can hide possible acyclic splits, as suggested by the previous example.
In fact, the example suggests considering only minimum implicational bases whose 
left-sides are as small as possible.
However, several such bases may exist and finding the right-one might be an 
expensive 
task, whence the following question.

\begin{question}
	Is it possible to decide whether a closure system has an acyclic split in 
	polynomial 
	time from an implicational base?
\end{question}

A similar question holds for the case of meet-irreducible elements:

\begin{question}
	Is it possible to recognize an acyclic split in polynomial time from a set of 
	meet-irreducible elements?
\end{question}

In Corollary \ref{cor:trad:charac-acyclic-from-meet}, we give a first step towards 
a 
characterization of acyclic splits from meet-irreducible elements.
The statement in Corollary \ref{cor:trad:charac-acyclic-from-meet} does consider 
the 
representation of closed sets by meet-irreducible elements.
Nonetheless, this characterization needs to be checked on every closed set of 
$\cs_2$.
In order to recognize an acyclic split from a set of meet-irreducible elements 
only, an 
idea would be to replace the statement by this one:
\begin{itemize}
	\item[] for every $M_2, M_2' \in \M_2$ such that $M_2 \subseteq M_2'$,
	$\Ext(M_2) \colon \U_1 \subseteq \Ext(M_2') \colon \U_1$.
\end{itemize}
Unfortunately, this latter condition is not sufficient, as demonstrated by the 
next 
example.

\begin{example}
	Let $\U_1 = \{4, 5\}$, $\U_2 = \{1, 2, 3\}$ and consider the closure systems 
	$\cs_1$ 
	and 
	$\cs_2$ given in Figure \ref{fig:trad:fail-base}.
	\begin{figure}[ht!]
		\centering 
		\includegraphics[scale=1.0, page=1]{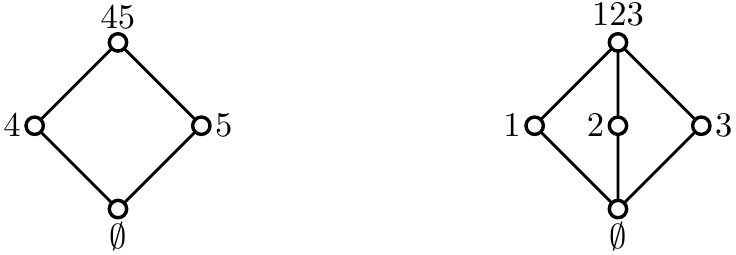}%
		\caption{The closure systems $\cs_1$ and $\cs_2$.}
		\label{fig:trad:fail-base}
	\end{figure}

	An implicational base for $\cs_1$ is $\is_1 = \emptyset$ and $\is_2 = \{12 
	\imp 3, 13 
	\imp 2, 23 \imp 1\}$ is an implicational base for $\cs_2$.
	We have $\M_1 = \{4, 5\}$ and $\M_2 = \{1, 2, 3\}$.
	Now let $\U = \U_1 \cup \U_2$ and consider the closure system $\cs$ of Figure 
	\ref{fig:trad:fail-cs} and the pair $(\U_1, \U_2)$.
	
	\begin{figure}[ht!]
		\centering 
		\includegraphics[scale=1.0, page=2]{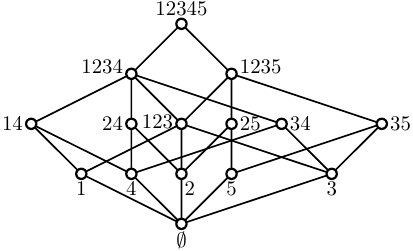}%
		\caption{The closure system $\cs$, failing Corollary 	
			\ref{cor:trad:charac-acyclic-from-meet}.}
		\label{fig:trad:fail-cs}
	\end{figure}
	
	We have $\M = \{1234, 1235 \} \cup \{14, 24, 25, 34, 35\}$.
	As $\M_2$ is an antichain, the condition given above is satisfied.
	However, Corollary \ref{cor:trad:charac-acyclic-from-meet} fails because 
	$\max(\Ext(\emptyset) \colon \U_1) \nsubseteq \Ext(1) \colon \U_1$.
	Hence, $(\U_1, \U_2)$ is not an acyclic split for $\cs$.
\end{example}

When $(\U_1, \U_2)$ is an acyclic split of $\cs$ and $\U_2$ is a singleton 
element, the 
construction of $\cs$ can be interpreted as the duplication of an ideal of $\cs_1$.
This puts the light on a possible link between (acyclic) splits and lower-bounded 
lattices \cite{freese1995free, adaricheva2013ordered}.
In particular, we know from \cite{adaricheva2013ordered} that the non-left-unit 
part of 
the $D$-base of a lower bounded lattice is acyclic.
As left-unit implications play no role in the existence of splits, there should 
exist a 
H-decomposition of the $D$-base by \emph{``almost acyclic''} splits.

\begin{example}
	Let $\U = \{1, 2, 3\}$ and $\is = \{12 \imp 3, 3 \imp 1\}$.
	The associated closure system is (lower) bounded.
	Its $D$-base is precisely $\is$.
	It has no acyclic split when we consider $3 \imp 1$, but it has a split $(12, 
	3)$ 
	which becomes acyclic once $3 \imp 1$ is removed.
\end{example}

Thus, we are naturally lead to the next question.

\begin{question}
	Can implicational bases of lower-bounded closure systems be characterized by 
	the 
	existence of a particular $\is$-tree? 
\end{question}

Answering this question would allow extending Theorem \ref{thm:trad:split-meet} to 
take 
into account unitary implications creating cycles.

\paragraph{Acknowledgments.} We are thankful to the reviewers for their helpful comments.
The second author is funded by the CNRS, France, ProFan project. 
This research is also supported by the French government IDEXISITE initiative 
16-IDEX-0001 (CAP 20-25).

\bibliographystyle{alpha}
\bibliography{biblio}

\begin{thebibliography}{BDVG18}

\bibitem[ADS86]{ausiello1986minimal}
Giorgio Ausiello, Alessandro D'Atri, and Domenico Sacca.
\newblock Minimal representation of directed hypergraphs.
\newblock {\em SIAM Journal on Computing}, 15(2):418--431, 1986.

\bibitem[AFN22]{adaricheva2022notes}
Kira Adaricheva, Ralph Freese, and James~B Nation.
\newblock Notes on join semidistributive lattices.
\newblock {\em International Journal of Algebra and Computation},
  32(02):347--356, 2022.

\bibitem[AN14]{adaricheva2014implicational}
Kira~V. Adaricheva and James~B. Nation.
\newblock On implicational bases of closure systems with unique critical sets.
\newblock {\em Discrete Applied Mathematics}, 162:51--69, 2014.

\bibitem[AN17]{adaricheva2017discovery}
Kira~V. Adaricheva and James~B. Nation.
\newblock Discovery of the {{D}}-basis in binary tables based on hypergraph
  dualization.
\newblock {\em Theoretical Computer Science}, 658:307--315, 2017.

\bibitem[ANR13]{adaricheva2013ordered}
Kira~V. Adaricheva, James~B. Nation, and Robert Rand.
\newblock Ordered direct implicational basis of a finite closure system.
\newblock {\em Discrete Applied Mathematics}, 161(6):707--723, 2013.

\bibitem[BDVG18]{bertet2018lattices}
Karell Bertet, Christophe Demko, Jean-Fran{\c c}ois Viaud, and Cl{\'e}ment
  Gu{\'e}rin.
\newblock Lattices, closures systems and implication bases: A survey of
  structural aspects and algorithms.
\newblock {\em Theoretical Computer Science}, 743:93--109, 2018.

\bibitem[Bic22]{bichoupan2022complexity}
Todd Bichoupan.
\newblock Complexity results for implication bases of convex geometries.
\newblock {\em arXiv preprint arXiv:2211.08524}, 2022.

\bibitem[BK13]{babin2013computing}
Mikhail~A. Babin and Sergei~O. Kuznetsov.
\newblock Computing premises of a minimal cover of functional dependencies is
  intractable.
\newblock {\em Discrete Applied Mathematics}, 161(6):742--749, 2013.

\bibitem[BK17]{babin2017dualization}
Mikhail~A. Babin and Sergei~O. Kuznetsov.
\newblock Dualization in lattices given by ordered sets of irreducibles.
\newblock {\em Theoretical Computer Science}, 658:316--326, 2017.

\bibitem[BM10]{bertet2010multiple}
Karell Bertet and Bernard Monjardet.
\newblock The multiple facets of the canonical direct unit implicational basis.
\newblock {\em Theoretical Computer Science}, 411(22-24):2155--2166, 2010.

\bibitem[BMN17]{beaudou2017algorithms}
Laurent Beaudou, Arnaud Mary, and Lhouari Nourine.
\newblock Algorithms for k-meet-semidistributive lattices.
\newblock {\em Theoretical Computer Science}, 658:391--398, 2017.

\bibitem[Das16]{dasgupta2016cost}
Sanjoy Dasgupta.
\newblock A cost function for similarity-based hierarchical clustering.
\newblock In {\em Proceedings of the Forty-Eighth Annual {{ACM}} Symposium on
  {{Theory}} of {{Computing}}}, pages 118--127, 2016.

\bibitem[DF12]{doignon2012knowledge}
Jean-Paul Doignon and Jean-Claude Falmagne.
\newblock {\em Knowledge Spaces}.
\newblock {Springer Science \& Business Media}, 2012.

\bibitem[DLM92]{demetrovics1992functional}
J{\'a}nos Demetrovics, Leonid Libkin, and Ilya~B. Muchnik.
\newblock Functional dependencies in relational databases: A lattice point of
  view.
\newblock {\em Discrete Applied Mathematics}, 40(2):155--185, 1992.

\bibitem[DN20]{defrain2020dualization}
Oscar Defrain and Lhouari Nourine.
\newblock Dualization in lattices given by implicational bases.
\newblock {\em Theoretical Computer Science}, 814:169--176, 2020.

\bibitem[DNV21]{defrain2021translating}
Oscar Defrain, Lhouari Nourine, and Simon Vilmin.
\newblock Translating between the representations of a ranked convex geometry.
\newblock {\em Discrete Mathematics}, 344(7):112399, 2021.

\bibitem[DP02]{davey2002introduction}
Brian~A. Davey and Hilary~A. Priestley.
\newblock {\em Introduction to Lattices and Order}.
\newblock {Cambridge university press}, 2002.

\bibitem[DS11]{distel2011complexity}
Felix Distel and Bar{\i}{\c s} Sertkaya.
\newblock On the complexity of enumerating pseudo-intents.
\newblock {\em Discrete Applied Mathematics}, 159(6):450--466, 2011.

\bibitem[Dun95]{dung1995acceptability}
Phan~Minh Dung.
\newblock On the acceptability of arguments and its fundamental role in
  nonmonotonic reasoning, logic programming and n-person games.
\newblock {\em Artificial intelligence}, 77(2):321--357, 1995.

\bibitem[EG95]{eiter1995identifying}
Thomas Eiter and Georg Gottlob.
\newblock Identifying the minimal transversals of a hypergraph and related
  problems.
\newblock {\em SIAM Journal on Computing}, 24(6):1278--1304, 1995.

\bibitem[EJ85]{edelman1985theory}
Paul~H. Edelman and Robert~E. Jamison.
\newblock The theory of convex geometries.
\newblock {\em Geometriae dedicata}, 19(3):247--270, 1985.

\bibitem[ENR21]{elaroussi2021lattice}
Mohammed Elaroussi, Lhouari Nourine, and Mohammed Radjef.
\newblock Lattice point of view for argumentation framework.
\newblock 2021.

\bibitem[FJ86]{farber1986convexity}
Martin Farber and Robert~E. Jamison.
\newblock Convexity in graphs and hypergraphs.
\newblock {\em SIAM Journal on Algebraic Discrete Methods}, 7(3):433--444,
  1986.

\bibitem[FJN95]{freese1995free}
Ralph Freese, Jaroslav Je{\v z}ek, and James~B. Nation.
\newblock {\em Free Lattices}, volume~42.
\newblock {American Mathematical Soc.}, 1995.

\bibitem[FK96]{fredman1996complexity}
Michael~L. Fredman and Leonid Khachiyan.
\newblock On the complexity of dualization of monotone disjunctive normal
  forms.
\newblock {\em Journal of Algorithms}, 21(3):618--628, 1996.

\bibitem[GD86]{guigues1986familles}
Jean-Louis Guigues and Vincent Duquenne.
\newblock Familles minimales d'implications informatives r\'esultant d'un
  tableau de donn\'ees binaires.
\newblock {\em Math\'ematiques et Sciences humaines}, 95:5--18, 1986.

\bibitem[Gr{\"a}11]{gratzer2011lattice}
George~A. Gr{\"a}tzer.
\newblock {\em Lattice Theory: Foundation}.
\newblock {Springer Science \& Business Media}, 2011.

\bibitem[GW12]{ganter2012formal}
Bernhard Ganter and Rudolf Wille.
\newblock {\em Formal Concept Analysis: Mathematical Foundations}.
\newblock {Springer Science \& Business Media}, 2012.

\bibitem[HK95]{hammer1995quasi}
Peter~L. Hammer and Alexander Kogan.
\newblock Quasi-acyclic propositional {{Horn}} knowledge bases: Optimal
  compression.
\newblock {\em IEEE Transactions on knowledge and data engineering},
  7(5):751--762, 1995.

\bibitem[HN18]{habib2018representation}
Michel Habib and Lhouari Nourine.
\newblock Representation of lattices via set-colored posets.
\newblock {\em Discrete Applied Mathematics}, 249:64--73, 2018.

\bibitem[JYP88]{johnson1988generating}
David~S Johnson, Mihalis Yannakakis, and Christos~H Papadimitriou.
\newblock On generating all maximal independent sets.
\newblock {\em Information Processing Letters}, 27(3):119--123, 1988.

\bibitem[Kha95]{khardon1995translating}
Roni Khardon.
\newblock Translating between {{Horn}} representations and their characteristic
  models.
\newblock {\em Journal of Artificial Intelligence Research}, 3:349--372, 1995.

\bibitem[KKS93]{kautz1993reasoning}
Henry~A. Kautz, Michael~J. Kearns, and Bart Selman.
\newblock Reasoning with characteristic models.
\newblock In {\em {{AAAI}}}, volume~93, pages 34--39. {Citeseer}, 1993.

\bibitem[KLS12]{korte2012greedoids}
Bernhard Korte, L{\'a}szl{\'o} Lov{\'a}sz, and Rainer Schrader.
\newblock {\em Greedoids}, volume~4.
\newblock {Springer Science \& Business Media}, 2012.

\bibitem[KN10]{kashiwabara2010characterizations}
Kenji Kashiwabara and Masataka Nakamura.
\newblock Characterizations of the convex geometries arising from the double
  shellings of posets.
\newblock {\em Discrete mathematics}, 310(15-16):2100--2112, 2010.

\bibitem[KSS00]{kavvadias2000generating}
Dimitris~J. Kavvadias, Martha Sideri, and Elias~C. Stavropoulos.
\newblock Generating all maximal models of a {{Boolean}} expression.
\newblock {\em Information Processing Letters}, 74(3-4):157--162, 2000.

\bibitem[Kuz04]{kuznetsov2004intractability}
Sergei~O. Kuznetsov.
\newblock On the intractability of computing the duquenne-guigues base.
\newblock {\em Journal of Universal Computer Science}, 10(8):927--933, 2004.

\bibitem[Lib93]{libkin1993direct}
Leonid Libkin.
\newblock Direct product decompositions of lattices, closures and relation
  schemes.
\newblock {\em Discrete Mathematics}, 112(1-3):119--138, 1993.

\bibitem[LLRK80]{lawler1980generating}
Eugene~L. Lawler, Jan~K. Lenstra, and AHG Rinnooy~Kan.
\newblock Generating all maximal independent sets: {{NP}}-hardness and
  polynomial-time algorithms.
\newblock {\em SIAM Journal on Computing}, 9(3):558--565, 1980.

\bibitem[LO78]{lucchesi1978candidate}
Claudio~L. Lucchesi and Sylvia~L. Osborn.
\newblock Candidate keys for relations.
\newblock {\em Journal of Computer and System Sciences}, 17(2):270--279, 1978.

\bibitem[Mar75]{markowsky1975factorization}
George Markowsky.
\newblock The factorization and representation of lattices.
\newblock {\em Transactions of the American Mathematical Society},
  203:185--200, 1975.

\bibitem[MR92]{mannila1992design}
Heikki Mannila and Kari-Jouko R{\"a}ih{\"a}.
\newblock {\em The Design of Relational Databases}.
\newblock {Addison-Wesley Longman Publishing Co., Inc.}, 1992.

\bibitem[MR94]{mannila1994algorithms}
Heikki Mannila and Kari-Jouko R{\"a}ih{\"a}.
\newblock Algorithms for inferring functional dependencies from relations.
\newblock {\em Data \& Knowledge Engineering}, 12(1):83--99, 1994.

\bibitem[Nat00]{nation2000unbounded}
James~Bryant Nation.
\newblock Unbounded semidistributive lattices.
\newblock {\em Algebra and Logic}, 39(1):50--53, 2000.

\bibitem[OD07]{obiedkov2007attribute}
Sergei~A. Obiedkov and Vincent Duquenne.
\newblock Attribute-incremental construction of the canonical implication
  basis.
\newblock {\em Annals of Mathematics and Artificial Intelligence},
  49(1):77--99, 2007.

\bibitem[Sho86]{shock1986computing}
Robert~C. Shock.
\newblock Computing the minimum cover of functional dependencies.
\newblock {\em Information Processing Letters}, 22(3):157--159, 1986.

\bibitem[TVL84]{tarjan1984worst}
Robert~E. Tarjan and Jan Van~Leeuwen.
\newblock Worst-case analysis of set union algorithms.
\newblock {\em Journal of the ACM (JACM)}, 31(2):245--281, 1984.

\bibitem[Wil94]{wild1994theory}
Marcel Wild.
\newblock A theory of finite closure spaces based on implications.
\newblock {\em Advances in Mathematics}, 108(1):118--139, 1994.

\bibitem[Wil95]{wild1995computations}
Marcel Wild.
\newblock Computations with finite closure systems and implications.
\newblock In {\em International Computing and Combinatorics Conference}, pages
  111--120. {Springer}, 1995.

\bibitem[Wil00]{wild2000optimal}
Marcel Wild.
\newblock Optimal implicational bases for finite modular lattices.
\newblock {\em Quaestiones Mathematicae}, 23(2):153--161, 2000.

\bibitem[Wil17]{wild2017joy}
Marcel Wild.
\newblock The joy of implications, aka pure {{Horn}} formulas: Mainly a survey.
\newblock {\em Theoretical Computer Science}, 658:264--292, 2017.

\bibitem[Zan15]{zanuttini2015proprietes}
Bruno Zanuttini.
\newblock Sur des propri\'et\'es structurelles des formules de horn.
\newblock In {\em 9es Journ\'ees d'{{Intelligence}} Artificielle Fondamentale
  ({{IAF}} 2015)}, 2015.

\end{thebibliography}

\end{document}